\documentclass[12pt]{amsart}
\usepackage{amssymb}
\usepackage{amsmath}
\usepackage{hyperref}
\usepackage{bbm}
\usepackage{xcolor}
\usepackage{cleveref}

\definecolor{refkey}{gray}{.5}   
\definecolor{labelkey}{gray}{.5} 
\usepackage{graphicx}
\usepackage{tikz-cd}
\usepackage{epsfig}
\usepackage{mathtools,amsfonts}

\usepackage[title]{appendix}

\usepackage{hyperref}

\usepackage[normalem]{ulem}

\usepackage{comment}
\usepackage{enumerate}
\DeclareMathOperator{\interior}{Int}

\newcommand{\RN}[1]{%
  \textup{\uppercase\expandafter{\romannumeral#1}}%
}

\def\bsigma{{\boldsymbol{\sigma}}}

\theoremstyle{plain}

\newtheorem{theorem}{Theorem}[section]
\newtheorem{lemma}[theorem]{Lemma}

\newtheorem{proposition}[theorem]{Proposition}

\newtheorem{hypothesis}[theorem]{Hypothesis}
\theoremstyle{definition}
\newtheorem{definition}[theorem]{Definition}
\newtheorem{remark}[theorem]{Remark}

\numberwithin{equation}{section}

\newcommand{\Mod}[1]{\ (\mathrm{mod}\ #1)}

\def\DS{\displaystyle}
\def\EXP{\mathbb{E}}

\def\brdelta{{\bar\delta}}
\def\brsigma{{\hat\sigma}}

\def\ba{{\mathbf{a}}}
\def\bE{{\mathbf{E}}}
\def\bsigma{{\boldsymbol{\sigma}}}

\newcommand{\I}{\mathcal{I}}
\newcommand{\R}{\mathbb{R}}
\newcommand{\Z}{\mathbb{Z}}
\newcommand{\T}{\mathbb{T}}

\newcommand{\eps}{\varepsilon}

\renewcommand{\phi}{\varphi}

\usepackage{amsmath}
\DeclareMathOperator{\lip}{Lip}

\DeclareMathOperator{\dist}{dist}
\DeclareMathOperator{\var}{Var}

\def\DS{\displaystyle}

\def\ba{{\mathbf{a}}}
\def\bE{{\mathbf{E}}}
\def\bsigma{{\boldsymbol{\sigma}}}

\def\DS{\displaystyle}
\def\EXP{\mathbb{E}}

\def\brW{{\bar W}}
\def\brr{{\bar r}}
\def\brT{{\bar T}}
\def\brP{{\bar P}}
\def\brG{{\bar G}}

\def\br\tau{{\bar tau}}

\def\cN{{\mathcal{N}}}
\def\cP{{\mathcal{P}}}
\def\cS{{\mathcal{S}}}

\def\cZ{{\mathcal{Z}}}

\def\fa{{\mathfrak{a}}}

\def\fa{{\mathfrak{a}}}
\def\RmII{I\!\!I}

\def\tG{{\tilde G}}

\def\mes{{\mathrm{mes}}}

\begin{document}

\author{Dmitry Dolgopyat}
\author{Davit Karagulyan}

\title[Dynamical random walk on the integers with a drift]{Dynamical random walk on the integers with a drift}
\thanks{We would like to thank Carlangelo Liverani for discussions on the manuscript.\\
The first author is partially supported by the NSF. The second author is supported by the Knut
and Alice Wallenberg foundation (KAW). The work was done during the second author’s visit to
the University of Maryland, and he thanks the mathematics department for the excellent working
conditions.
}
\maketitle{}

\begin{abstract}

In this note we study dynamical random walks (DRW) with internal states. We consider a particle which performs a dynamical random walk on $\mathbb{Z}$ and whose local dynamics is given by expanding maps. 
We provide sufficient conditions for the position of the particle $z_n$ to satisfy the Central Limit Theorem.

\end{abstract}

\tableofcontents

\section{Introduction} 
\subsection{Motivation.}
Understanding transport in an inhomogeneous media is one of the classical problems in mathematical physics.
The motion in homogeneous media is well understood and is described by the heat equation whose fundamental solution is given by the transition density of the Brownian Motion. The situation in the inhomogenuous case
is more complicated.

One of the simplest models of inhomogenuous transport is given by random walks in random environment. In this 
model the particle moves on the lattice $\Z^d$ so that if the particle is in position $z$ it moves to $z+v$ for $v$ in a finite set
$\Lambda$ with probability $p(z, v)$ where the vectors $\{p(z, \cdot)\}_{z\in\Z^d}$ are iid.
This model is completely understood in dimension 1 while in higher dimensions only partial results are available.
In the one dimensional setting the recurrent motion leads to Sinai behavior \cite{S82}, where the particle at time $t$
is typically at the distance $O\left(\ln^2 t\right)$ from the origin. In the transient case a wide range of behaviors
is possible \cite{KKS}. In particular the transient walk can have either positive or zero speed
(\cite{Sol}). 
 In the case of positive 
speed the fluctuations around the linear motion could be either Gaussian or described by stable laws of index
$1\leq s<2.$ In the case of zero speed the limit distributions are Mittag--Leffler (the results of \cite{KKS, S82, Sol} 
pertain to the nearest neighbor walks, we refer the readers to \cite{BG00, BG08, DG13, DG20, Go08} 
for the extensions to the walks with bounded jumps). In contrast if the dimension is greater than 1, then the walk
is expected to satisfy the Central Limit Theorem (at least, if the dimension is high enough). However, so far it has
been proven only for systems satisfying some additional assumptions such as reversibility (\cite{PV, Bi11}),
a sufficiently strong drift (see \cite{Szn01, Szn02, BDR} and references wherein) or a perturbative regime
(\cite{SznZ}). 

The progress in understanding of random walks in random environment naturally leads to a question about
extending the results proven for that model to a more realistic systems. One particularly interesting question
is to understand a deterministic motion in random environment. In particular, a number of papers concern 
Lorentz gas in random environment--a system, where a particle moves freely on a plane colliding elastically
with a random array of convex scatterers (\cite{AL18, CL10, DL21, LT11}).
While the works above  establish  recurrence and the Law of Large Numbers for different models of random
Lorentz gas, the limit theorems are currently unknown. In order to obtain a more tractable model of 
deterministic motion in random environment, in \cite{AL18} the authors proposed a model of Deterministic Walks 
in Random Environment (DWRE). 
By this one means a map $F$ defined on $M\times \Z^d$
where $M$ is the internal state of the walker.
Namely,  suppose that for each $n\in \Z^d$ we have a map $T_n: M\to M$ and a partition
$\DS M=\bigcup_v W_{v, n}$ ({\em gate partition}) where $v\in \{0, \pm e_1,\dots \pm e_d\}.$  Let
\begin{equation}
\label{DRW} F(x, n)=\Big(T_n x, n+\sum_v 1_{W_{v,n}} v\Big). 
\end{equation}
Thus if the particle is at site $n$ then its internal state changes according to $T_n$, while the change of the location 
is prescribed by the gates.

One is then interested in statistical properties of $z_n(x)=\pi_{\mathbb{Z}^d}(F^n(x,0))$. The randomness in the system comes from the random choice of the initial internal state $x_0\in M$.

In \cite{AL18} the authors provide conditions under which $z_n$ satisfies the law of large numbers. They
 show that their conditions are satisfied for a dynamical random walk whose local dynamics is given by 
a sufficiently expanding interval map, such as 
$\beta$ transformations with large $\beta.$

{ \cite{AL18} show that the random Lorenz gas fits into the framework of DWRE. Moroeover the class}
 of DWRE contains several classical examples of random motion.
As an example, consider the following system:
let $d=1$, $M=\T^{1}$, $T_n(x)=2x \Mod 1$ and $W_{n,-1}=[0, \frac{1}{2})$, $W_{n,1}=[\frac{1}{2}, 1)$ for all $n \in \mathbb{Z}$. One can see that  if we choose the initial internal state uniformly on $\mathbb{T}$ then
the DWRE defined this way is equivalent to the 
{ simple symmetric random walk on $\Z$.}
More generally it is shown in \cite{AL18}
that DWRE with { linear}
expanding local dynamics and Markov gates can model random walks in random 
environment (RWRE). In particular, all types of behavior  observed in RWRE, appear also
in DWRE, so the particle can be transient with zero speed (\cite{Sol})
 or it can exhibit Sinai 
behavior (\cite{S82}) where after $n$ steps the particle is located at the distance of order $\ln^2 n$
from the origin. However, Markov condition on the gates is pretty restrictive and so it is of interest
to develop tools to handle non Markovian dynamics.

The goal of the present article is to develop a robust method for proving CLT for one dimensional systems with
strong drift 
(note that some assumptions on the system are necessary to get the CLT due to the non-Gaussian examples
of \cite{AL18}). Our approach has two types of ingredients: probabilistic and dynamical. The dynamical ingredient 
is the CLT theory for the composition of ladder maps $G_n$. $G_n$ describes the internal state of the particle 
starting at the level $n$
when it arrives at level $n+1$ for the first time.
This part relies on the theory
of sequential dynamical systems. The probabilistic ingredients consist of renewal theory which allows to pass from
the CLT for hitting times to the CLT for the particle position and on the CLT for the quenched drift, which uses the
central limit theory for weakly dependent random variables.

In order to describe the main ideas of our approach in the simplest possible settings we
present two models. Model A is strongly ballistic. Namely among any three steps, at least two are to the right. 
In this case the dynamical part uses the CLT for bounded observables of sequential expanding maps
available in the literature (\cite{CR07}). 
Model B is more realistic, since the particle could move arbitrary far to the left, albeit with a small probability.
In this case the dynamical part needs to be extended as well leading to more complicated arguments.

In a future work we plan to apply our method to Lorentz gas in the presence of random field. In this case
the local dynamics and, hence, the ladder maps $G_n$ are hyperbolic rather than expanding which requires
a significant improvement of the existing dynamical results. Therefore this model will be a subject of a separate
paper.

We note that our work is the first example, where the CLT is proven for an open class 
of deterministic systems
in random environment as time tends to infinity
({the results of \cite{AL18} could be used to obtain examples
of DRWE satisfying the  CLT, however, the parameters need to be tuned very carefully to obtain the equivalence 
with RWRE)}.
Before describing precisely our models (which will be done
in \S \ref{SSRes}) we 
{ mention the previous work where the CLT is obtained for the deterministic motion in random environment. We note that in the models described below the environment
depends on an additional parameter $\eps$ and the time scales
as some power of $1/\eps$, while in our work the environment is fixed and time tends to infinity.}

 The first model deals with a
particle moving in a dilute random media (so called Boltzmann-Grad regime) where the time tends to zero and the
sizes of the scatters go to zero at the same time. A selection of papers on this subject includes
\cite{BBS, Sp80, Sp91, LuT}. While this topic is of great physical relevance, it is beyond the scope of
the present work.
We just mention that since the interactions happen rarely, it is easier to  
make use of the mixing properties of the environment. 

{ Another problem dealing with  deterministic motion in the random media is equations with rapidly oscillating coefficients.
The study of the equation $\dot{x}=v(x, \xi_{t/\eps})$, where $\xi_t$ is a rapidly mixing random process 
and $\varepsilon>0$ is a small parameter, goes back
to the work of Khasminskii (\cite{Kh1, Kh2}). Note that this system is non autonomous but it can be converted to an 
autonomous form by rewriting it as
$$ \dot{x}=v(x, \xi_{s/\eps}), \quad \dot{s}=1.$$
{ Khasminskii shows that the solutions of this equations are close to the solutions
of the averaged equation $\dot{\bar x}=\bar v(\bar x)$ where $\bar v(\bar x)=\EXP(v(\bar x, \xi))$
and obtains the CLT for the fluctuations.}
More generally, the results similar to \cite{Kh1, Kh2} can be obtained for the systems in the form
$\dot{x}=v(x/\eps)$ where $v:\R^d\to\R^d$ is a rapidly mixing process and $v_{d}>\delta$ for some $\delta>0.$
(In this case $x_d$ plays the role of time), see \cite{KP79}. 
Similar results are also available for the the second order equations with rapidly oscillating coefficients,
see \cite{KP81, KR1, KR2, DK09} and references wherein.}

A third subject is the billiard models where the reflections
from the boundary are random to model microscopic roughness of the walls
(see \cite{CPSV09, CPSV10a, CPSV10b, Fe07}). 
While limit theorems are available for the random model, the derivation of the same laws from the underlying 
microscopic dynamics remains a challenging open problem.


\subsection{Results.}
\label{SSRes}
We consider a model of DRW { defined by \eqref{DRW}}
where $d=1,$ $M=\T$ and $T_n: \T\to \T$ are smooth uniformly expanding maps.
We will also assume that the particle's coordinate changes every time, thus
$\DS \T=W_{n, -1}\cup W_{n, 1}$. We consider the following models. 

{\em Model A.} 
Let $\brT$ be an expanding map so that there are constants
 $3<\gamma\leq K$ and $K_1>0$, such that
for all $n$ and all $x\in \T$ we have
\begin{equation}
\label{BGeom}
\gamma\leq |\brT_n'(x) |\leq K, \quad \sup_{x \in \mathbb{T}}|\brT_n''(x)|\leq K_1.
\end{equation}

Let $\brW \subset \mathbb{T}$ be a segment such that 
\begin{equation}\label{balist}
\brT^p \brW \cap W=\emptyset \text{ for }p=1, 2.
\end{equation}
We also suppose that for a sufficiently small $\delta_0$ we have that for all $n$
$\|T_n -\brT \|_{C^{2}(\mathbb{T})}\leq \delta_0$ and the Hausdorff distance between $W_{n,-1}$ and $\brW$ is 
smaller than $\delta_0.$

Note that
 the condition $\brT^p \brW \cap W=\emptyset \text{ for }p=1, 2$, is a \emph{ballisticity} condition 
 ensuring that among every three moves of the particle at least two are to the right. Thus the particle moves 
to the right ballistically. Namely $z_n\geq n/3$ and $z_n\geq z_{m}-1$ for $n>m$. 
\\

{\em Model B.} { Let $a$ be a large integer.}
Let $\brT(x)=ax \Mod 1$ and $\brW_{-1} \subset \mathbb{T}$ be a segment with $|\brW_{-1}|<\frac{1}{2}$.
We suppose that for a sufficiently small $\delta_0$ and for all $n\in \mathbb{Z}$,
$\|T_n -\brT \|_{C^{2}(\mathbb{T})}\leq \delta_0$ and the Hausdorff distance between $W_{n,-1}$ and $\brW$ is 
smaller than $\delta_0.$

Thus in this model the local dynamics enjoys a strong expansion, which makes this model similar to the one considered
in \cite{AL18}. (Note that in Model B the local dynamics is smooth while \cite{AL18} consider $\beta$ transformations
which have discontinuity on the circle. We believe that the method of our paper can be extended to maps
with a finite number of discontinuities provided that slope is sufficiently large (depending on the number
of discontinuity points) but to keep the presentation simple we restrict our attention to smooth maps.)
\\

Our first result is the CLT for the hitting time. Namely let $\tau_n({ x}, k)$ be the smallest time $t$
such that $F^t({ x}, k)\in \T\times \{n\}.$ Define the maps
$G_n: \T\to \T$ by 
\begin{equation}\label{induce}
G_n(x)=\pi_{\T} F^{r_{n}(x)}(x, n) \text{ where }
r_n(x)=\tau_{n+1}(x, n)
\end{equation}
and $\pi_{\T}$ denotes the projection on the first coordinate.
{Thus $G_n(x)$ describes the internal state of the walker, which starts at site $n$ with internal state 
$x$, at the first time when the walker reacher site $n+1$.
}
We shall also write $\tau_n(x):=\tau_n(x, 0).$ 
Note that
\begin{equation}\label{hitting}
\tau_n(x)=\sum_{k=0}^{n-1} r_k(G_{k-1}\circ\dots \circ G_0 x). 
\end{equation} 

We say that the DRW satisfies the CLT for hitting times if

\begin{equation}\label{lmt}
\frac{\tau_n-\EXP(\tau_n)}{\sqrt{\var(\tau_n)}}\Rightarrow \mathcal{N}(0,1)
\text{ as }
n \rightarrow \infty
\end{equation}
where $\cN(a, \sigma^2)$ denotes the normal distribution with mean $a$ and standard deviation
$\sigma.$ Here and elsewhere in this article we assume (unless it is explicitly stated otherwise)
that $x$ is uniformly distributed on $\T.$

\begin{theorem}
\label{ThHit}

(a) Given $\bar T$ there exist $\brdelta_0$ such that if 
 $\delta_0\leq \brdelta_0$  then 
the DRW from model A satisfies the CLT for hitting times.

(b) Assume that in model B, $|\brW_{-1}|< \frac{1}{2}$.
Then there exists $a_0 \in \mathbb{N}$ so that for all $a \in \mathbb{N}$, with $a \geq a_0$,  there exists $\bar{\delta}(a)$ so that if $\delta\leq \bar{\delta}(a)$ then the maps $\{G_n\}_{n \in \mathbb{Z}}$ are well defined and 
the DRW satisfies the CLT for hitting times.
\end{theorem}

\begin{remark}
We can assume that $|\brW_{-1}|<\frac{1}{2}$ without loss of generality. Otherwise, we will switch $\brW_{-1}$ with $\brW_{1}$.
\end{remark}

In order to obtain some information about the position of the particle $z_n$ we need to choose the maps and the gates in an iid way. 

Let $\mathcal{E}=\{(T_1,W_1), \dots, (T_m, W_m)\}$ be a collection of maps and gates, so that any sequence $(T_n,W_{n,-1})_{n \in \mathbb{Z}}$, with $(T_n, W_n) \in { \mathcal{E}}$, $\forall n \in \mathbb{Z}$, satisfies the conditions of Theorem \ref{ThHit} (that is either for all $n$ the assumptions of model A are satisfied, 
or for all $n$ the assumptions of model  B are satisfied).

\begin{theorem}\label{main}
Take $\bar{\delta}$ so small that every realization $\{(T_n, W_n)\}_{n \in \mathbb{Z}}$ from the collection $\mathcal{E}$ satisfies the conditions of Theorem \ref{ThHit}. 

  (a) {\sc (Quenched CLT)} There are constants $\ba,\bsigma>0$ such that for almost all iid realizations
  of the pairs $(T_n,W_n)$ there are constants $b_n=b_n(\omega)$ such that if 
  $x$ is uniformly distributed on $\T$ 
  then
  $$
\frac{z_n-b_n}{(1/{\ba^{3/2}})\bsigma \sqrt{n}}\Rightarrow \mathcal{N}(0,1)
\text{ as }
n \rightarrow \infty. $$
(b) {\sc (Annealed CLT)}There are constants $v, \sigma$ such that 
if $x$ and $ \{(T_n, W_n)\}$ are independent, $x$ is uniformly distributed on $\T$ 
and $(T_n, W_n)$ are chosen from $\mathcal{E}$ in an iid fashion, then
  $$
\frac{z_n-v n}{(1/{\ba^{3/2}})\sigma \sqrt{n}}\Rightarrow \mathcal{N}(0,1)
\text{ as }
n \rightarrow \infty. $$
\end{theorem}

\begin{remark}
The ballisticity condition \eqref{balist} ensures that the time needed to move to the right 
{ for Model A}
is in $BV$ as a function 
of the initial condition $x.$ This allows us to apply existing results about the central limit theorem for 
non-autonomous dynamical systems, such as \cite{AR14, CR07, DFGTV, HNTV, NSV12}.  
In case the return time is unbounded, as is the case for Model B, 
one needs to extend the existing result allowing much less
regular functions. This extension is formulated in Theorem \ref{pr-1} and is proven in the appendix.
This result is of independent interest.

We  also hope that our approach will be useful for other models
of motions in random media, and this will be a subject of a future work.
\end{remark}

\section{Notations and definitions}
\label{ScND}

For the sequence of maps $\{G_n\}_{n \in \mathbb{Z}}$ and $k\leq m$ we define maps $G_{k,m}$ as follows
$$
G_{k,m}(x)=
G_{m} \circ\cdots \circ G_{k}(x).
$$
If above we have only one map, i.e. $G_m=G$, for all $m \in \mathbb{Z}$, then 
$G_{k,m}=G^{m-k+1}$.

Let
$$
D^{(2)}_n=\{(t_1, t_2, \cdots, t_{n}):t_k \in \{-1,1\}, 1\leq k \leq n\}.
$$

\begin{definition}
For $n \geq 1$, $t \in D_n^{(2)}$, let
$$
s_k(t)=t_1 + \dots + t_k, 
\quad 1 \leq k \leq n,\quad\text{and}\quad
s_0(t)=0.
$$
\end{definition}

We will also be interested in the following subset of $D_n^{(2)}$
\begin{definition}\label{r-def}
Let $R_n^{(2)} \subset D_n^{(2)}$, be the set of all $t=(t_1, \dots, t_n)$, for which
$$
s_k(t)
\leq 0, \quad 1\leq k \leq n-1,
\quad
\text{and}
\quad
s_n(t)= 1.
$$
\end{definition}
\begin{remark}
For even $n$ we have $R_n^{(2)}=\emptyset$.
\end{remark}

\begin{definition}\label{I-def}
Let $A \subset \mathbb{T}$ be a set such that there is a collection of closed and disjoint intervals $\{I_k\}_{k=1}^m$ so that $\cup_{k=1}^m I_k\subset \bar A$ and $|A \setminus\cup_{k=1}^n I_k|=0$. Denote
$$
\I(A)=\{I_1, \dots, I_n\}.
$$
\end{definition}
Note that if $\I(A)$ exists, then it is unique. 

\medskip

We now recall some definitions and facts from \cite{CR07}. Denote by $BV$ the space of all functions with bounded variation and by $V (f)$ the variation of the function $f \in BV$. The space $BV$ is equipped
with the norm
$$
|f|_{BV} := V (f) + ||f||_1,
$$
where $||f||_1$ is relative to the Lebesgue measure. For $f \in BV$, we have: $||f||_\infty \leq |f|_{BV}$.

Define also $BV_0 = \{f \in  BV : \int_\mathbb{T} f dx = 0\}$. \smallskip

We will be interested in maps satisfying

\begin{hypothesis}\label{hyp}
$G:\mathbb{T}\rightarrow \mathbb{T}$ is such that there exists a finite or countable partition $(I_j)$ of $[0,1]$ or $\mathbb{T}$ such that the restriction of the map  to each interval $I_j$ is strictly monotone  and $G\vert_{I_j} \in C^2(I_j)$. 
We also assume that
$$
{ \gamma \quad :=\inf _{j} \inf _{x \in I_{j}}\left|G^{\prime}(x)\right|>2,
\quad\mathrm{and}\quad
\sup _{j} \sup_{x \in I_j}\left|\frac{(G(x))''}{(G(x)')^2}\right|<\infty.}
$$
\end{hypothesis}
{ Given a map $G$ as above} define
$$
K :=\sup_{j} \sup_{x \in I_j}\left|G'(x)\right|\text{ and }K_1 :=\sup _{j} \sup_{x \in I_j}\left|G''(x)\right|.
$$

Note that we can  have $K, K_1=\infty$.

\begin{definition}\label{markov}
We say that
{a collection of intervals $\{I_k\}_{k=1}^m$, with $\DS \bigcup_{k=1}^m I_k= \mathbb{T}$
is a Markov partition for a}
map $G$ satisfying Hypothesis \ref{hyp}, 
if
$$
G(I_k)=\mathbb{T}, \quad 1\leq k \leq m,
$$
and $G$ is injective and continuous on each $I_k$. 
\end{definition}

The transfer operator of a map satisfying Hypothesis \ref{hyp} is given by
\begin{equation}\label{deftrans-1}
P_G f(x)=\sum_{j} f\left(\sigma_{j} x\right) \frac{1}{\left|G^{\prime}\left(\sigma_{j} x\right)\right|} 1_{G\left(I_{j}\right)}(x),
\end{equation}
where $\sigma_j$ is the inverse function of the restriction of $G$ on $I_j$. 
It is well known that
$$
\int_\mathbb{T} \left(P_{G} f \right) g d x=\int_\mathbb{T} f(x) g(Gx) d x, \;\;\forall f \in L^{1}, g \in L^{\infty}.
$$
Note also the following form of the transfer operator
$$
P_G(f)(x)=\sum_{y: G(y)=x}\frac{f(y)}{\left|G^{\prime}(y)\right|}.
$$
\begin{lemma}\label{l2-norm}
Let $f \in L^2(\T)$. Then
$$
\Big\|P f\Big\|_2 \leq \sqrt{|P \mathbf{1}|_\infty} \; \|f\|_2.
$$
\end{lemma}
\begin{proof}
By H\"older's inequality
$$
\Big(\sum_{y: G(y)=x}\frac{f(y)}{\left|G^{\prime}(y)\right|}\Big)^2\leq \Big(\sum_{y: G(y)=x} \frac{f^2(y)}{|G'(y)|}\Big) \Big(\sum_{y: G(y)=x} \frac{1}{|G'(y)|}\Big). 
$$
Integrating this inequality we obtain
$$
\int_\mathbb{T} (Pf)^2 dx \leq \int_\mathbb{T}Pf^2dx |P \mathbf{1}|_\infty=|P \mathbf{1}|_\infty \int_\mathbb{T}f^2dx.
$$
Taking square root on both sides we get the required estimate.
\end{proof}

Let $\mathcal{P}$ be a set of contractions on $L^1$ (a set of linear operators satisfying $\|Pf\|_1\leq \|f\|_1$, for every $P\in \mathcal{P}$).
Following \cite{CR07} the $p$ distance between two transfer operators $R,R'$ will be defined as follows
\begin{equation}
\label{DefDP}
d_p\left(R, R^{\prime}\right)=\sup _{\left\{f \in BV :|f|_{BV} \leq 1\right\}}\left\|R f-R^{\prime} f\right\|_p.
\end{equation}
When $p=1$, we will drop the index and denote it by $d$.
For $P \in \mathcal{P}$, we denote its $\delta$ neighborhood by $B(P, \delta):=\{R \in \mathcal{P}: d(R, P)<\delta\}$.

\smallskip

We say that the collection $\mathcal{P}$ satisfies the Lasota-Yorke
property $\bold{(LY)}$, if there exists $\rho \in (0,1)$ and $C>0$, so that for any $P \in \mathcal{P}$ we have
\begin{equation}\label{LY}
\forall f \in BV, \quad V(P f) \leq \rho V(f)+C\|f\|_1\tag{$\bold{LY}$}.\end{equation}

We say the subset $\mathcal{P}_0 \subset \mathcal{P}$ satisfies the exponential decay of correlations property $\bold{(Dec)}$ in $BV_0$ if there exist $\theta < 1$ and
$K > 0$ such that, for all integers $l \geq 1$, all $l$-tuples of operators $P_1, \dots , P_l$
in $\mathcal{P}_0$ we have
\begin{equation}\label{dec}
\forall f \in BV_0, \quad |P_l\cdots P_1f|_{BV} \leq  K \theta^l
|f|_{BV}\tag{$\bold{Dec}$}.
\end{equation}
It follows from \eqref{LY} (see \cite[Lemma 2.4]{CR07}) that there exists $M>0$, so that for any $P_n, \dots, P_1 \in \mathcal{P}$ and $f \in BV$
\begin{equation}\label{den-bdd}
\left|P_{n} \cdots P_{1} f\right|_{BV} \leq M|f|_{BV},\quad \text{ for all } n \geq 1.
\end{equation}

We say that the sequence of operators $\{P_n\}_{n \geq 1}$ satisfies the condition $\bold{(Min)}$, if there exists $\sigma>0$ such that 
\begin{equation}\label{Min}
P_nP_{n-1}\dots P_1 \mathbf{1}(x)\geq \sigma, \quad \forall x \in \mathbb{T},\forall n \in \mathbb{N}\tag{$\bold{Min}$}.
\end{equation}
We say that the collection $\mathcal{P}$ satisfies conditions \eqref{Min} if any sequence in $\mathcal{P}$ satisfies the property \eqref{Min} with the same constant $\sigma$.
In the sequel we will use the notation 
$$
\mathcal{P}^{n} \mathbf{1} = P_nP_{n-1}\cdots P_1 \mathbf{1}.
$$

We recall a criterion for verifying the condition \eqref{dec}:


\begin{proposition}(\cite[Proposition 2.10]{CR07})
\label{nbrhd-sc}
Let $\mathcal{P}$ be a collection of contractions satisfying \eqref{LY} and $P \in \mathcal{P}$ that satisfies \eqref{dec}, i.e.
\begin{equation}\label{spc-gap}
|P^nf|_{BV} \leq C \gamma^n |f|_{BV}, \quad \forall f \in BV_0.
\end{equation}
Then there exists $\delta_{0}>0$, such that the set $\mathcal{P}_{0}=B\left(P, \delta_{0}\right)\cap \mathcal{P}$ satisfies \eqref{dec} in $BV_{0}$
\end{proposition}

The relevance of the properties introduced above comes from the following result.

\begin{theorem}{\cite[Theorem 5.1]{CR07}}
\label{CRThm}
Let $(f_n)$ be a sequence of observables, so that 
$\DS \sup_{n\geq 1} |f_n|_{BV}<\infty$. Assume that for the sequence of transformations $\{T_n\}_{n \geq 1}$ the corresponding set of transfer operators $\{P_{T_n}\}_{n \geq 1}$ satisfy  \eqref{Min} and \eqref{dec}. Let
$$
S_n(x) =\sum_{k=0}^{n-1}f_n(T_{1,n}(x))-\int_{\mathbb{T}}f_n(T_{1,n}(x))dx
$$ 
where $T_{1,n}=T_n\circ \dots \circ T_1$.
If the norms $\|S_n\|_2$ are unbounded as $n \rightarrow \infty$ then
$$
\frac{S_{n}}{\|S_n\|_2} \Rightarrow \mathcal{N}(0,1).
$$
\end{theorem}

Theorem \ref{CRThm} is sufficient to handle Model A. For Model B we need an extension of this result,
namely, Theorem \ref{pr-1} { formulated in \S \ref{SSConzeRaugi}} 
and proven in Appendix \ref{app}. 
Theorem \ref{pr-1} allows to handle unbounded observable and  is of independent interest.

\begin{definition}
We say that the observable $\phi$ is cohomologous to zero for the map $T$ if there exist an observable $\mathbf{H}\in L^2$ and $c\in \mathbb{R}$ such that
$$
\phi +c = \mathbf{H} - \mathbf{H}\circ T.
$$
\end{definition}

\section{Some auxiliary results}
\label{ScAux}

\subsection{An extension of a result of Conze-Raugi}
\label{SSConzeRaugi}
In Appendix \ref{app} we prove the following extension of Theorem \ref{CRThm}. Observe that the functions $\{f_n\}$ below can also be unbounded

\begin{theorem}\label{pr-1}
Assume the operators $\{P_n\}_{n\geq 0}$ fulfill the conditions \eqref{dec} and \eqref{Min} on $BV$ and let $\{f_n\}_{n \geq 1}$ be a sequence of observables, such that there exists $D>0$, so that
\begin{equation}\label{norm-bd}
\sup _{\left\{g \in BV :|g|_{BV} \leq 1\right\}}|P_n(f_{n-1}g)|_{BV}<D \hbox{ and }
\sup _{\left\{g \in BV :|g|_{BV} \leq 1\right\}}|P_n(f^2_{n-1}g)|_{BV}<D,
\end{equation}
for all $n \geq 1$. Consider the sum
$$
S_n(x) =\sum_{k=0}^{n-1}\tilde{f}_k,
\text{ where } 
\tilde{f}_k=f_k(G_{1,k}(x))-\int f_k(G_{1,k}(x))dx. 
$$
If the sequence of variances $\sigma_n=\|S_n\|_2$ is unbounded and for every $\varepsilon>0$ we have
\begin{equation}\label{asm}
\lim_{n \rightarrow \infty}\frac{\sum_{k=1}^n \int \tilde{f}_{k}^{2}(x) 1_{\left[\varepsilon \sigma_{n}, \infty\right)}\left(\tilde{f}_{k}^{2}(x)\right)dx
}{\sigma^2_n}=0,
\end{equation}
then
$\DS 
\frac{S_n}{\sigma_n} \Rightarrow \mathcal{N}(0,1)
$
as $n \rightarrow \infty$.
\end{theorem}

{
\begin{remark}
Let $\T_n=\T\times \{n\}.$ Note that in several results in Sections \ref{ScND} and \ref{ScAux} including Theorems \ref{CRThm} 
and \ref{pr-1} we consider maps $G_{1,k}$ with domain $\T_1$ and range $\T_{k+1}.$ 
In particular $f_k$ are defined on $\T_{k+1}.$
This is done to have the same notation as in \cite{CR07}.
However, in applications we will deal with maps $G_{0, k-1}$ 
with domain $\T_0$ and range $\T_{k}.$  This is done since it is natural to consider the walk started at the origin 
rather than site 1.
\end{remark}
}

\subsection{Lasota-Yorke inequality.}
We need the following standard fact whose proofs could be 
found in \cite{CR07}, page 105.
\begin{lemma}
\label{LmBoundary-Var}
(a) Let $[u, v] \subset[c, d] \subset[0,1]$, and $f$ be of bounded variation. Then
\begin{equation}\label{est-1}
|f(u)|+|f(v)|
\leq V_{[c, d]}(f)+\frac{2}{(d-c)} \int_c^d |f(t)| d t.
\end{equation}

(b) In particular
$$ |f(u)|+|f(v)|
\leq V_{[u, v]}(f)+\frac{2}{(v-u)} \int_u^v |f(t)| d t. $$
\end{lemma}

\begin{lemma}\label{lz-prf}
Let $G$ satisfy Hypothesis \ref{hyp} and suppose that there is an interval $W\subset\T$ such that
$T$ is smooth everywhere except, possibly, at the endpoints of $W.$ Assume also there is $K>0$ such that 
$\DS \sup_{x \in \mathbb{T}\setminus \partial W}|G'(x)|\leq K$
Then there is $C=C(\gamma, K, K_1)>0$ such that
\begin{equation}
V(P_G f) \leq \frac{3}{\gamma} \mathrm{V} (f)+C\|f\|_1 .
\end{equation}
\end{lemma}
\begin{proof}
Recall that
$$
P_{G} f(x)=\sum_{j} f\left(\sigma_{j} x\right) \frac{1}{\left|G^{\prime}\left(\sigma_{j} x\right)\right|} 1_{G\left(I_{j}\right)}(x),
$$
where $\sigma_j$ is the inverse function of $G$ on its intervals of monotonicity $(I_j)$.
We can assume that $|W|\leq\frac{1}{2}$. Otherwise, instead of $W$ we can consider $W^c$. 
Since $G$ has only two discontinuity points, the partition $(I_j)$ can be chosen in such a way that there will be at most one interval $I \in (I_j)$, with $|I|<\frac{1}{2([K]+1)}$.
Indeed, we can define the partition $(I_j)$ on $W^c$ so that $|I_j|=\frac{|W^c|}{[K]+1}$,  $j=1, \dots, [K]+1$. Since $\frac{1}{2([K]+1)}<|I_j|<\frac{1}{K}$, then $G|_{I_j}$ will be one-to-one on each one of these intervals. If now $|W|<\frac{1}{2([K]+1)}$, then we will take $W$ to be one of the partition intervals, otherwise we divide $W$ into intervals of size $\frac{1}{2([K]+1)}$ and a reminder interval $I_i$, so that $|I_i|<\frac{1}{(2[K]+1)}$.

Note that
$$
V(P_G f)\leq \sum_{j} V\Big(f\left(\sigma_{j} x\right) \frac{1}{\left|G^{\prime}\left(\sigma_{j} x\right)\right|} 1_{G\left(I_{j}\right)}\Big)
\leq $$
\begin{equation}
\label{VProd}
\sum_{j}\left( V_{G(I_j)}\left[\left(\frac{f}{G^{\prime}}\right) \circ \sigma_{j}\right]
+\Big[\Big|\frac{f}{G^{\prime}}\Big|(\sigma_{j} \alpha_{j})+\Big|\frac{f}{G^{\prime}}\Big|
(\sigma_{j} \beta_{j})\Big]\right)=:  \text{I}+\text{II}
\end{equation}
where $G(I_j)=[\alpha_j, \beta_j].$ 
By an inequality in \cite{CR07}, page 106, we have 
$$ V_{G(I_j)}\left[\left(\frac{f}{G^{\prime}}\right) \circ \sigma_{j}\right]=
V_{I_j}\left[\left(\frac{f}{G^{\prime}}\right) \right]
\leq \frac{V_{I_j}(f)}{\gamma}+\frac{K_1}{\gamma^2} \int_{I_j} |f(t)| dt. $$
Summing over $j$ we get
\begin{equation}
\label{1G-1}
\text{I}\leq \frac{V(f)}{\gamma}+\frac{K_1}{\gamma^2}\|f\|_1. 
\end{equation} 
Next for all  monotonicity intervals with $|I_j|>\frac{1}{(2[K]+1)}$
we use Lemma \ref{LmBoundary-Var}(b) obtaining
$$ \Big|\frac{f}{G^{\prime}}\Big|(\sigma_{j} \alpha_{j})+
\Big|\frac{f}{G^{\prime}}\Big|(\sigma_{j} \beta_{j})
\leq\frac{1}{\gamma} \left[|f|(\sigma_{j} \alpha_{j})+
|f|(\sigma_{j} \beta_{j})\right]$$
\begin{equation}
\label{1G-2}
\leq 
\frac{V_{I_j}(f)}{\gamma}+\frac{2}{\gamma|I_j|}\int_{I_j}  |f(x)| dx
\leq \frac{V_{I_j}(f)}{\gamma}+\frac{4([K]+1)}{\gamma} \int_{I_j}  |f(x)| dx.
\end{equation}

It remains to handle the shortest interval $I_i.$
Let $I_{i+1}$ be a partition element adjacent to $I_i$ and set 
$I= I_i \cup I_{i+1}$.  Then $|I|>\frac{1}{2([K]+1)}$ and by \eqref{est-1}, applied to $I_i\subset I=[c,d]$, we have
$$
\Big|\left(\frac{f}{G^{\prime}}\right)\left(\sigma_{i} \alpha_{i}\right)\Big|+\Big|\left(\frac{f}{G^{\prime}}\right)\left(\sigma_{i} \beta_{i}\right)\Big| \leq
$$
\begin{equation}
\label{1G-3}
 \frac{1}{\gamma}(|f(\sigma_i \alpha_i )+|f|(\sigma_i \beta_i)) 
 \leq \frac{1}{\gamma} V_{I}(f) + \frac{2}{\gamma |I|}\int_{I} |f|dx.
\end{equation}
Summing the above estimates we obtain
$$
V(P_{G}f)
 \leq \frac{3}{\gamma}V(f) + C\|f\|_1
$$
where the factor $\frac{3}{\gamma}$ is the sum of three terms of size $\frac{1}{\gamma}$
coming from \eqref{1G-1}, \eqref{1G-2}, and \eqref{1G-3} respecively.
This completes the proof.
\end{proof}

\subsection{Positivity of density.}\label{dens}

We say that the sequence of expanding maps $\{G_n\}_{n \geq 1}$ satisfies  \emph{property }(C) if for every $\varepsilon>0$, there exists $s\geq 1$ and $M>0$ such that for every $x\in \mathbb{T}$, $n \in \mathbb{N}$ and any interval $I\subset\mathbb{T}$, with $|I|>\varepsilon$, there exists $y=y(n,I,x) \in I$ so that
$$
G_{n, n+s}(y)=x, \quad |D_yG_{n, n+s}|\leq M.
$$ 
\begin{lemma}\label{rmcov}
Assume there exists $K>0$, such that  for every $n \geq 1$, $$ \sup_{j}\sup_{x \in I^{(n)}_j}|G_n'(x)|\leq K.$$ 
{If $\{G_n\}$ satisfy 
 the classical covering property 
 namely for every interval $I$ there exist a number $s \in \mathbb{N}$, such that for every $n \geq 1$,\;
$\DS 
G_{n,n+s}(I)=\mathbb{T}
$
then property (C) holds.}
\end{lemma}
\begin{proof}
Take a partition of $\mathbb{T}$ into intervals $\{J_k\}$ of lengths in $[\frac{\varepsilon}{4} ,\frac{\varepsilon}{2}]$. For each $J_k$ we can find its own covering number $s_k$. Note that any number $s$ larger than $s_k$ is again a covering number for $J_k$. Let $s$ be the largest number in the set $\{s_k\}$. Now observe that any interval $J$ of length larger than $\varepsilon$ contains an interval from $\{J_k\}$ in its interior. Hence we will have $G_{n,n+s}(I)=\mathbb{T}$. It remains to notice that 
$|D_x G_{n,n+s}|\leq K^s.$
\end{proof}

We say the map $G$ satisfies property (C) if the sequence $G, G, \dots$ satisfies this property.

The following proposition extends 
several classical results for a single expanding map (see \cite{Liv1}) to 
{ a sequence of} expanding maps satisfying property (C).

\begin{proposition}\label{sptgp}
Let $\mathcal{G}=\{G_n\}_{n \geq 1}$, be a sequence of expanding maps so that for each $n \geq 1$ there is an interval $J_n \subseteq \mathbb{T}$ such that $G_n(J_n)=\mathbb{T}$ and $\DS \sup_{n \geq 1}\sup_{x \in J_n}|G_n'(x)|\leq K_0$, for some finite $K_0$. Assume also the set of associated transfer operators $\{P_n\}_{n \geq 1}$ satisfies property \eqref{LY}. Then
\begin{enumerate}
  \item If $\mathcal{G}$ satisfies property (C) then there exists $\sigma>0$ so that for any $n \geq 1$ and $x \in \mathbb{T}$
$$
P_nP_{n-1}\cdots P_1\mathbf{1}(x)\geq \sigma.
$$

  \item Let $G\in \mathcal{G}$ be an expanding map which satisfies property (C). Then $P$ also satisfies property \eqref{dec}.

\end{enumerate}

\end{proposition}

\begin{proof}
(a)
We follow the proof of \cite[Proposition 2]{AR14}. 
For $a>0$, let
$$
\mathcal{E}_{a}=\left\{f \in BV : f \geq 0, V(f) \leq a \int f\right\}.
$$
By Lemma 3.2 in \cite{Liv1}, for any $f \in \mathcal{E}_{a}$ there exist an interval $I$, with $|I|=\frac{1}{2a}$, so that $f(x) \geq \frac{1}{2} \int f$ for all $x \in I$.
Note that by the Lasota-Yorke inequality \eqref{LY} we have
\[
V(P_{r} \cdots P_{1} f) \leq \rho^{r}V(f)+C_{r}\|f\|_{1} \leq\left(a \rho^{r}+C_{r}\right) \int f.
\]
Hence, for $a \geq \frac{C_{r}}{1-\rho^{r}},$ we have $\left(P_{r} \ldots P_{1}\right)\left(\mathcal{E}_{a}\right) \subset \mathcal{E}_{a}$ for any choice of $P_{1}, \ldots, P_{r} .$ In
order also to have $\mathbf{1} \in \mathcal{E}_{a},$ we will actually choose $a=\max \left\{1, \frac{C_{r}}{1-\rho_{r}}\right\}$.
By Property (C), for $\varepsilon=1/2a$ we can find $s\geq 1$, $M<\infty$  so that for any $x \in \mathbb{T}$ and any $n \in \mathbb{Z}$ one can find $\zeta \in I$, with $G_{n,{n+s}}(\zeta)=x$ and 
$\DS
|D_{\zeta}G_{n,{n+m}}| \leq M.
$

Let $m \geq 0$. We have $P_{{1},{m+s}} \mathbf{1}=P_{{m+1},{m+s}} P_{1,m} \mathbf{1} .$ Write $m=p_{m} r+q_{m},$ with $0 \leq q_{m}<r$.
Note that for $k\leq q_m$ we have
\[
\quad P_{{1},{k}} \mathbf{1}(x)\geq K_0^{-k},
\]
since all the maps $G_{n}$ have intervals $J_n$ so that $G_n(J_n)=\mathbb{T}$  and $|G_n'(x)|\leq K_0$, for all $x \in J_n$. As a consequence, we have $P_{{1},{m+s}} \mathbf{1} \geq K_0^{-q_{m}} P_{{m+1},{m+s}} g_{m},$ with $g_{m}=P_{q_{m}+1,m} \mathbf{1}$. Since $P_{q_{m}+1,m} \mathbf{1}$ is a concatenation
of $p_{m}$ blocks of $r$ operators applied to a function in $\mathcal{E}_{a},$ we obtain that $g_{m}$ belongs to $\mathcal{E}_{a}$. Then, there exists an interval $I$, with $|I|=\frac{1}{2a}$, on which $g_{m} \geq \frac{1}{2} .$ This implies
$$
P_{{1},{m+s}} \mathbf{1}(x) \geq K_0^{-q_m} P_{{m+1},{m+s}} \mathbf{1}_{I}(x)
$$
$$
=K_0^{-q_m} \sum_{G_{{m+1},{m+s}}(y)=x} \frac{\mathbf{1}_{I}(y)}{\left|\left(D_y G_{{m+1},{m+s}}\right)\right|} 
\geq \frac{K_0^{-q_m}}{M}
$$
completing the proof.

(b) By part (a) for a proper choice of the parameter $a$ we have that $P^r \left(\mathcal{E}_{a}\right) \subset \mathcal{E}_{a}$. 
{ Given this inclusion the proof of \eqref{dec} can be made
the same way as in \cite[Section 3]{Liv1} or \cite[\S 3.2]{Viana} so we omit it.}
\end{proof}

\begin{lemma}\label{min-1}
Let $\{G_n\}_{n \geq 1}$ be a sequence of maps and intervals $\{W_n\}_{n \geq 1}$ such that for each $n \geq 1$ $G_n$ is continuous everywhere on $\mathbb{T}$, except possibly at the endpoints of $W_n$. Assume also we have $|G_n(x)'|\leq K$ at all points $x$ away from discontinuity points. Then there exists $\sigma>0$ such that
$$
\left(P_n\cdots P_1 1\right)(x)\geq \sigma, \quad \forall x \in \mathbb{T}, n\geq 1,
$$
where $P_n$ is the transfer operator for $G_n$.
\end{lemma}
\begin{proof}
We verify the conditions of Proposition \ref{sptgp}(a). First note that since $\gamma>3$ then one can find an interval $J_n\subset W_n$ or $J_n\subset W_n^c$ so that $G_n(J_n)=\mathbb{T}$ and we obviously have $|D_x G_n|\leq K <\infty$. Property \eqref{LY} follows from Lemma \ref{lz-prf}. To verify property (C), in view of \Cref{rmcov}, it is sufficient to show that $\{G_n\}_{n\geq 1}$ satisfies the covering property: for each $I \subset \mathbb{T}$, there exists $N=N(|I|)$ so that $G_{1,N}(I)=\mathbb{T}$. 

If $W_1\cap I \neq \emptyset$, 
then the intersection with $\partial W_1$ divide $I$ into at most three components.
Let $I_1$ be the largest component.
Then $|I_1|\geq |I|/3$. Consider the image $G_1(I_1)$. Then $|G_1(I_1)|\geq \gamma |I_1|$, since $G_1$ is continuous both inside and outside of $W_1$. Next, we choose the largest interval $I_2 \subset G_1(I_1)$, so that either $I_2 \subseteq W_2$ or $I_2 \cap W_2= \emptyset$. Hence, $|I_2|\geq |G_1(I_1)|/3 > \frac{\gamma}{3}|I_1|$. Repeating this argument, we will obtain a sequence of intervals $(I_n)_{n \geq 1}$, so that
$$
|I_{n+1}|> \Big(\frac{\gamma}{3}\Big)^n |I_{n}|.
$$
Since $\frac{\gamma}{3}>1$, the image of $I_1$ covers the circle in
time $O\left(\ln \left(\frac{1}{|I_1|}\right)\right)$.
\end{proof}

\section{The growth of variance.}
\label{ScVarB}

In this section we study the behavior of the variance of $\tau_n$
$$
\sigma_{n}^{2}=\int_{\mathbb{T}}\left(\sum_{i=1}^{n}\left[r_{i}\left(G_{1, i}(x)\right)-\int_{\mathbb{T}} r_{i}\left(G_{1, i}(y)\right) d y\right]^{2}\right) d x.
$$

The next proposition shows that the linear growth of variance is stable under small perturbations. Note that the observables $r_n$ and $r$ may  be unbounded.

{Recall \Cref{DefDP}.}
\begin{proposition}\label{lm-1}
Let $\mathcal{G}$ be a collection of maps satisfying Hypothesis \ref{hyp} such that its associated set of transfer operators satisfies \eqref{dec}. 
Assume $G \in \mathcal{G}$, $r\in L^2(\mathbb{T})$ are such that the acim $h$ of $G$ is bounded away from zero and $r$ is not cohomologous to a constant for $G$. Let $P$ be the transfer operator of $G$.
{ Then for each $L>0$ there exists $\delta_0>0$ such that the following holds.
Let} $\{G_n\}_{n \geq 1}\subset \mathcal{G}$ 
and $r_n\in L^2(\mathbb{T})$  be such that 
{ denoting  by $P_n$ the transfer operators of $G_n$ we have that}
for  all $n \geq 1$
\begin{equation}\label{gnrl-r}
|P_n (r_{n-1} f)|_{BV}\leq L |f|_{BV}, \quad |\brP (\brr f)|_{BV}\leq L |f|_{BV} ,\quad f \in BV.
\end{equation}
{ and}
 $$d_2(P_n, P)\leq \delta_0, \quad d_2(P_n(r_{n-1} \cdot),  P(r \cdot))\leq \delta_0, \quad
  \|r - r_n\|_2\leq \delta_0.$$ 
Then 
$$
\sigma_n^2=\var(\tau_n) \geq Cn,
$$
where $C=C(\delta_0, \brr, \brG, { L})>0$.
\end{proposition}

\begin{proof}

Define
$$
\tilde{r}_k=r_k - \int_\mathbb{T} r_k(G_{1,k}(x))dx.
$$

By assumption $r$ is not cohomologous to zero and $h(x)\geq c>0$ for almost all $x \in \mathbb{T}$ and some $c>0$. Then by Proposition \ref{A-2} {proven in the appendix} there exists $C>0$ so that
\begin{equation}\label{vrns}
\bar \sigma^2_n \geq Cn,
\end{equation}
where
$$
\bar\sigma_n^2 = n\sum_{i=1}^n \int_\mathbb{T} \tilde{r}^2 hdx + 2 \sum_{k=1}^n(n-k)\int_\mathbb{T} \tilde{r} (x)\tilde{r} (G^k(x))dx
$$
is the variance of the unperturbed system.

Similarly for the general case
\begin{equation}\label{kobe}
\sigma_n^2 = \sum_{i=1}^n \int_\mathbb{T} \tilde{r}^2_{i} (G_{{1},i} (x))dx + 2 \sum_{1 \leq i<j \leq n}\int_\mathbb{T} \tilde{r}_{i} (G_{1,{i}}(x) )\tilde{r}_{j} (G_{1,{j}}(x))dx.
\end{equation}
We now show that for each $\varepsilon>0$, $\delta_0$ can be taken so small that for all large $n\geq 1$
$$
|\bar\sigma_n^2 - \sigma_n^2|\leq \varepsilon n.
$$
To this end we note that
$$
\Big|\int_\mathbb{T} \tilde{r}_{i}\left(G_{1, i} (x)\right) \tilde{r}_{j}\left(G_{1, j} (x)\right) d x \Big|=
\left|\int_\mathbb{T} \tilde{r}_{i} P_{i} \cdots P_{j+1}\left(\tilde{r}_{j} \mathcal{P}^{j} 1\right) d x\right| \leq
$$
$$
\leq K \theta^{|i-j|}\left|P_{j+1}(\tilde{r}_{j} \mathcal{P}^{j} 1)\right|_{BV}\left\|\tilde{r}_{i}\right\|_{1} \leq D^{\prime} \theta^{|i-j|},
$$
where in the last line we used  \eqref{dec} and the estimate 
\begin{equation}\label{intr-1}
|P_{j+1}(\tilde{r}_{j} \mathcal{P}^{j} \mathbf{1})|_{BV}\leq |P_{j+1}(r_{j} \mathcal{P}^{j} \mathbf{1})|_{BV} +\left|\int_\mathbb{T}r_j(G_{1,j})dx\right||\mathcal{P}^{j+1} \mathbf{1}|_{BV}
\end{equation}
$$
\leq L |\mathcal{P}^j\mathbf{1}|_{BV} + \|r_j\|_1 \|\mathcal{P}^j \mathbf{1}\|_\infty |\mathcal{P}^{j+1} \mathbf{1}|_{BV}\leq L M + L M^2,
$$
which relies on the fact that $\|P_{j+1}(r_j)\|_1=\|r_j\|_1\leq L$.
Therefore 
$$
\Big|\sum_{i,j \leq n; |i-j|\geq N}\int_\mathbb{T} \tilde{r}_i(G_{1,i}x)\tilde{r}_{j} (G_{1,{j}}x)dx\Big|\leq 
\sum_{N \leq i\leq n}(n-i)D' \theta^{i}
\leq 
D'\sum_{N \leq i\leq n}n\theta^{i}
\leq D' n \frac{\theta^N}{1-\theta}.
$$
In a similar way for $\bar \sigma_n$ we will have
\begin{equation}\label{large}
\left|2 \sum_{k=1}^n (n-k)\int_\mathbb{T} \tilde{r} (x)\tilde{r} (G^k(x))h(x)dx\right| \leq D^{\prime} n \frac{\theta^{N}}{1-\theta}.
\end{equation}
Next, we take $N$ so large that
\begin{equation}\label{large1}
D^{\prime} n \frac{\theta^{N}}{1-\theta}\leq n\frac{C}{4},
\end{equation}
where $C$ is from \eqref{vrns}.
We now consider the terms with $|i-j|<N$ and show that for arbitrary $\varepsilon>0$, $\delta_0$ can be taken so small that the following bound holds
\begin{equation}\label{dist-var}
\Big|\int_\mathbb{T} \tilde{r}_{i} P_{i} \cdots P_{j+1}(\tilde{\tau}_{j} \mathcal{P}^{j} 1) d x-\int_\mathbb{T} \tilde{r}(x) P^{|i-j|}(\tilde{r}h) dx\Big|\leq C_0 \varepsilon,
\end{equation}
for some $C_0>0$.
For this it is enough to show that for arbitrary $\varepsilon>0$, $\delta_0$ can be taken so small that if $\delta\leq  \delta_0$ then
\begin{equation}\label{exp-1}
\|\tilde{r}_{i}-\tilde{r}\|_2\leq \varepsilon,
\end{equation}
and
\begin{equation}\label{exp-2}
\left\|P_{i} \cdots P_{j+1}(\tilde{r}_{j} \mathcal{P}^{j} \mathbf{1})-P^{|i-j|}(\tilde{r}h)\right\|_2 \leq \varepsilon.
\end{equation}

By Lemma 2.13 of \cite{CR07}, for any $p\leq n$ we have that
\begin{equation}
\left\|\mathcal{P}^{n} \mathbf{1}-P^{n} \mathbf{1}\right\|_1 \leq  C'\left(p \delta_0+\left(1-\theta\right)^{-1} \theta^{p}\right).
\end{equation}
Taking $p=[\frac{1}{\sqrt{\delta_0}}]+1$, we see that for small $\delta_0$
\begin{equation}\label{min-3}
\left\|\mathcal{P}^{n} \mathbf{1}-P^{n} \mathbf{1}\right\|_1 \leq  C'\sqrt{\delta_0}.
\end{equation}
Since, $P^{n} \mathbf{1} \rightarrow_{L^1} h$, as $n\to\infty$, then for $n$ sufficiently large $\left\|\mathcal{P}^{n} \mathbf{1}-h\right\|_1 \leq  2C'\sqrt{\delta_0}$. Thus
$$
\left\|\mathcal{P}^{n} \mathbf{1}-h\right\|_2 \leq \sqrt{2M\left\|\mathcal{P}^{n} \mathbf{1}-h\right\|_1} \leq C_1\delta_0^\frac{1}{4}.
$$
Next, observe that
\begin{equation}\label{mean-dist}
\Big|\int_\mathbb{T} (r_k(G_{1,k})-r(G^k))dx \Big|=\Big|\int_\mathbb{T} (\mathcal{P}^k \mathbf{1}r_k-P^k \mathbf{1} r	) dx \Big|
\end{equation}
$$
\leq \Big|\int_\mathbb{T} \mathcal{P}^n \mathbf{1}(r_k-r) +  (\mathcal{P}^n \mathbf{1}-P^n \mathbf{1}) r dx \Big|
$$
$$
\leq M \|r_k-r\|_1 +  \|\mathcal{P}^n \mathbf{1}-P^n \mathbf{1}\|_2 \|r\|_2 \leq M \delta_0 + \|r\|_2 L^{\frac{1}{2}}\delta_0^{\frac{1}{4}}. 
$$
It then follows that for all $n \geq 1$ and $\delta,\varepsilon$ small we will have
\begin{equation}\label{sddd}
\|\tilde{r}-\tilde{r}_n \|_2 \leq M \delta_0 + \|r\|_2 L^{\frac{1}{2}}\delta_0^{\frac{1}{4}} + \|r-r_n\|_2
\leq (M+1)\delta_0 + \|\brr\|_2 L^{\frac{1}{2}}\delta_0^{\frac{1}{4}},
\end{equation}
since by assumption $\|r-r_n\|_2\leq \delta_0$.
Thus we obtain \eqref{exp-1}.

To show\eqref{exp-2}, observe that by the triangle inequality
$$
\Big\|P^{|i-j+1|}\Big(\tilde{r}h\Big) -  P_{i} \cdots P_{j+1}\left(\tilde{r}_j \mathcal{P}^{j} 1\right)  \Big\|_{2} \leq \Big\| P^{|i-j|}\Big(P \Big(\tilde{r}h\Big)\Big)-P^{|i-j-1|} \left(P_{j+1} (\tilde{r}_j \mathcal{P}^{j} 1)\right) \Big\|_{2}
$$
\begin{equation}\label{intrmd}
+\Big\|P^{|i-j|} \left(P_{j+1} (\tilde{r}_j \mathcal{P}^{j} 1)\right)-P_{i} \cdots P_j(P_{j+1}\left( \tilde{\tau}_j\mathcal{P}^{j} 1\right)) \Big\|_{2} =I+\RmII.
\end{equation}
By Lemma 2.4 of \cite{CR07} we have that
$$
d_2(P_1\dots P_n, P^n)\leq \sum_{k=1}^n d_2(P_k, P).
$$
Hence, by assumptions of the Proposition and in view of \eqref{intr-1}
$$
I \leq  d_2\left(P_{i} \cdots P_{j}, P^{|i-j-1|}\right)\left| P_{j+1}(\tilde{\tau}_j\mathcal{P}^{j} 1)\right|_{BV} 
\leq (LM + L M^2) \sum_{k=1}^{|i-j|} d_2\left(P_{k}, P\right)\leq NC_3 \delta_0.
$$
For $\RmII$ we have by Lemma \ref{l2-norm} and the assumptions of our proposition that
$$
\RmII \leq  \Big\| P^{|i-j-1|} \Big(P_{j+1} (\tilde{r}_j \mathcal{P}^{j} 1)-(P (\tilde{r}h)\Big)\Big\|_{2}.
$$
$$
\leq \|P \mathbf{1}\|_\infty^{|i-j-1|/2}\Big\|P_{j+1} (\tilde{r}_j \mathcal{P}^{j} 1)-P (\tilde{r}h)\Big\|_2 \leq C_4 M^{|i-j-1|/2}\delta_0.
$$
Taking $\delta_0$ small enough we arrive at \eqref{exp-2}.
Combining \eqref{exp-1} and \eqref{exp-2} we get \eqref{dist-var}. 
Summing \eqref{dist-var} for all $|i-j|\leq N$ we get
$$
\Big|\sum_{i,j\leq n, |i-j|\leq N}\int_\mathbb{T} \tilde{\tau}_{i} (G_{1,{i}}(x) )\tilde{\tau}_{j} (G_{1,{j}}(x))dx
-\sum_{k=1}^N \Big((n-k)\int_\mathbb{T}f(x)f(G^k(x))dx\Big)\Big|
$$
\begin{equation}\label{i-j}
\leq  C'Nn \varepsilon+ C''N \varepsilon.
\end{equation}
Thus by \eqref{i-j}, \eqref{large} and \eqref{large1} we can write 
$$
|\sigma_n^2 -  \bar \sigma_n^2|\leq n\frac{2C}{4} + C'Nn \varepsilon+ C''N \varepsilon.
$$
Therefore
$$
\sigma_n^2 \geq  Cn - 2\frac{C}{4}n-C'Nn \varepsilon- C''N \varepsilon
\geq \Big( \frac{C}{2}- C'N\varepsilon \Big)n  - C''N \varepsilon\geq C_1 n
$$
 if $\varepsilon$ is small enough. This finishes the proof.
\end{proof}

\begin{lemma}\label{cob}
Let $\brG$ be an expanding map satisfying Hypothesis \ref{hyp} so that for some $x_0 \in \mathbb{T}$ we have $\brG(x_0)=x_0$  and $\brG$ is continuous at a neighborhood of $x_0$. Assume for $\brr \in L^2$ we have $\brr = C$ in an open neighborhood of $x_0$ and $C \neq \int_\mathbb{T} \brr hdx$. Assume further that $\brP \brr \in BV$. Then $\bar{\tau}$ is not cohomologous to a constant under $\brG$. 
\end{lemma}
\begin{proof}
Assume $\brr$ is a coboundary for $\brG$. Since $\brP\brr \in BV$, then 
by Proposition \ref{A-2} there exists $g \in BV$, such that the equality
$$
\brr(x)-\int_\mathbb{T}\brr(x)h(x)dx =g(x)-g(\brG(x))
$$
holds almost surely. Let $A'\subset \mathbb{T}$ be the set of all $x \in A'$ for which the equation above holds for all the forward and backward images of $x$ under $\brG$. Clearly $|A'|=1$.

By assumption $\brG(x_0)=x_0$. Take $x \in A'$. Then 
$$
\sum_{k=0}^n\brr(\brG^k(x))-n\int_\mathbb{T}\brr(x)h(x)dx =g(x)-g(\brG^n(x)).
$$
Observe that $\brr(x)=\brr(x_0)$ for $x$ sufficiently close to $x_0$. Hence
\begin{equation}\label{cb}
|n\brr(x) - n\int_\mathbb{T}\brr(x)h(x)dx|
=n|C-\int_\mathbb{T}\brr(x)h(x)dx|\leq  2 \|g\|_\infty.
\end{equation}
By assumption $C-\int_\mathbb{T}\brr(x)h(x)dx\neq 0$
Hence, for large $n$, \eqref{cb} can not take place.
This finishes the proof.
\end{proof}

\section{Proof of the main results for Model A.}
\subsection{Proof of Theorem \ref{ThHit}(a).} 
\label{mod-b}

By assumption
$$
\brT^p \brW \cap \brW=\emptyset \text{ for }p=1, 2.
$$

Recall that $T_n=\brT+h_n$, where $h_n \in C^{2}(\mathbb{T})$, $\|h_n\|_{C^{2}}<\delta_0$ and $| \brW \triangle W_n|<\delta_0$. { By \eqref{balist}}, for $\delta_0$ sufficiently small we will have
$$
T_n(W_n)\cap W_{n-1}=\emptyset, \quad T_{n-1}(T_n(W_n))\cap W_{n}=\emptyset.
$$
This implies that
\begin{equation}\label{maps}
G_n(x) = \begin{cases}
T_{n}(T_{n-1}(T_{n}(x))) &\text{if}\quad x \in W_{n}, \\
T_{n}(x)&\text{if}\quad x \in \mathbb{T}\setminus W_{n}.
\end{cases}
\end{equation}
Hence the hitting times are
\begin{equation}\label{times}
r_{n}(x)=\begin{cases} 
3  & \text{if}\quad x \in W_{n}, \\
1 & \text{if}\quad \mbox{ for }x \in \mathbb{T}\setminus W_{n}.\\
\end{cases}
\end{equation}

Let $\brr$ and $\brG$ respectively be 
$$
\brr(x)=\begin{cases} 
3 &\text{if}\quad x \in \brW, \\
1 & \text{if}\quad x \in \mathbb{T}\setminus\brW;\end{cases} \quad
\brG(x) = \begin{cases}
\brT(\brT(\brT(x)) &\text{if}\quad x \in \brW, \\
\brT(x)&\text{if}\quad x \notin \mathbb{T}\setminus \brW.\\
\end{cases}
$$

We denote by $P_n$ and $\brP$ the transfer operators of $G_n$ and $\brG$ respectively.

To prove Theorem \ref{ThHit}(a)  we will verify that the collection $\{P_n\}$ satisfies the conditions of Theorem \ref{CRThm}.

Next, we  show that the sequence $\{P_n\}$ satisfies  \eqref{dec} 
if $\delta_0$ is sufficiently small. 
Note that 
\eqref{dec} for the unperturbed map $\brP$ follows from \Cref{sptgp}(b). 
Applying Proposition \ref{nbrhd-sc} to $\brP$, we can find a neighborhood of $\brP$, where \eqref{dec} property is preserved. Hence, to establish \eqref{dec} for the collection $\{P_n\}_{n \geq 1}$  for $\delta_0$ small, it suffices to show that the norms 
$d_1(P_n, \brP)$ are small when $\delta_0$ is small. 
Thus \eqref{dec} is a consequence of the following result whose proof will be given in \S \ref{SSModAd-dist}.

\begin{lemma}\label{d-dst}
For $\delta_0$ sufficiently small there exists $L>0$ such that
$$
\|P_nf - \brP f\|_1\leq L\delta_0 |f|_{BV}, \quad {\forall f} \in BV.
$$
\end{lemma}

The \eqref{Min} condition for $\mathcal{P}^n \mathbf{1}$ follows from Lemma \ref{min-1}. Thus, we have established \eqref{dec} and \eqref{Min} for the sequence $\{P_n\}$. 

Note that for Model A, 
$1<\int_\mathbb{T}\brr hdx<3$. 
Hence by Lemma \ref{cob}, $\brr$ is not cohomologous to a constant for $\brG$.
Thus, Proposition \ref{lm-1} 
 gives the linear growth of the variance for the sequence 
$$ \tau_n(x)=\sum_{k=0}^{n-1} r_k(G_{k-1}\circ\dots \circ G_0 x)$$
if $\delta_0$ sufficiently small.
Now Theorem \ref{ThHit}(a)
 follows from  Theorem \ref{CRThm}. \qed

\subsection{Proof of Lemma \ref{d-dst}}
\label{SSModAd-dist}
We recall the following fact from \cite{CR07}. Let
\begin{align} 
\widetilde{w}(f, t)=\int_{0}^{1} \sup _{|y-x| \leq t}|f(y)-f(x)| dx. 
\end{align}
Then
\begin{equation}\label{var-bd}
\widetilde{w}(f, t) \leq 2 t V(f).
\end{equation}
As earlier, we need to estimate the norms
$$
\|P_n f-\brP f\|_1=\int_\mathbb{T}|P_n f-\brP f|dx.
$$
For $x \in W_n \cap \brW$ we have that
$$
G_n(x)=T_n(T_{n-1}(T_n(x)))=\brT(\brT(\brT(x)))) + g_n(x),
$$
where $\|g_n\|_{C^{1+\lip}({W}_n\cap \brW)}<C\delta_0$. Hence, away from a set $A \subset \mathbb{T}$ of measure $O(\delta_0)$ we have $\|\brG - G_n\|_{C^{1 + \lip}(\mathbb{T}\setminus A)}<L\delta_0$, $\forall n \in \mathbb{Z}$.
As both $\brG$ and $G_n$ are continuous everywhere away from the endpoints 
of the intervals $\bar{W}$ and $W_n$, 
{then there is a set $B$ of measure $O(\delta_0)$ such that for each $x$ outside of $B$}
 for every preimage $y_1$, with $\brG^{-1}(x)=y_1$, there is a preimage $y_2$, $G_n^{-1}(x)=y_2$ close to 
 $y_1$. 
 Since $|G'_n|\leq K^3$, then for each $x$ there are at most $[K^3]+1$ many inverse branches of $G_n.$
One can also see that $|y_1-y_2|\leq L\delta_0/\gamma$. We now write
\begin{align*} \label{mm}
P_{G_n} f(x)-P_{\brG} f(x)&=E_0(x)  + \sum_{y_1:G_n(y_1)=x} \frac{f(y_1)}{\brG'(y_1)}-\sum_{y_2:\brG(y_2)=x} \frac{f(y_2)}{G'_n(y_2)}\\
&=E_0(x) + \sum_{y_1,y_2}\left[\frac{f(y_1)}{\brG'(y_1)}-\frac{f(y_2)}{\brG'(y_1)}\right] + \sum_{y_1,y_2}\left[\frac{f(y_2)}{\brG'(y_1)}-\frac{f(y_2)}{G'_n(y_2)}\right]
\end{align*}
where $E_0$ is supported on $B.$
In particular $\|E_0\|_1\leq C \delta_0$.
Next,
\begin{equation}\label{intr}
\left|\frac{1}{\brG^{\prime}(y_1)}-\frac{1}{G_n^{\prime}(y_2)}\right|\leq 
\left|\frac{K_1(y_1-y_2)}{\brG^{\prime}(y_1)\brG^{\prime}(y_2)}\right| + \left|\frac{1}{\brG^{\prime}(y_2)}-\frac{1}{G_n^{\prime}(y_2)}\right|
\leq L\left(\frac{K_1 \delta_0}{\gamma^3} + \frac{\delta_0}{\gamma^2}\right).
\end{equation}
Note that we have $K_1<\infty$. By the triangle inequality
$$
\int_{\mathbb{T}\setminus B}\left|\frac{f(y_1)}{\brG'(y_1)}-\frac{f(y_2)}{G'_n(y_2)}\right|dx\leq  \int_{\mathbb{T}\setminus B}\left|\frac{f(y_1)}{\brG'(y_1)}-\frac{f(y_2)}{\brG'(y_1)}\right|dx
+ \int_{\mathbb{T}\setminus B}\left|\frac{f(y_2)}{\brG'(y_1)}-\frac{f(y_2)}{G'_n(y_2)}\right|dx.
$$
For the first term on the right, we have by \eqref{var-bd}
\begin{align*}
\int_{\mathbb{T}\setminus B}\left|\frac{f(y_1)}{\brG'(y_1)}-\frac{f(y_2)}{\brG'(y_1)}\right|dx &\leq \frac{1}{\gamma}\int_{\mathbb{T}\setminus B}|f(y_1)-f(y_2)|dx \\
&\leq \frac{2}{\gamma} \sup|y_1-y_2|V_{\mathbb{T}}(f)\leq \frac{2L\delta_0}{\gamma^2}V_{\mathbb{T}}(f).
\end{align*}
By \eqref{intr}
$$
\int_{\mathbb{T}\setminus B}\left|\frac{f(y_2)}{\brG'(y_1)}-\frac{f(y_2)}{G'_n(y_2)}\right|dx\leq \|f\|_\infty \delta_0 L\left(\frac{K_1}{\gamma^3}+\frac{1}{\gamma^2}\right).
$$
Since the measure of  $B$ is of order $O(\delta_0)$, we also have
$$
\int_B |P_nf - \brP f| dx\leq L_1 \|f\|_\infty |B|\leq L_2\delta_0 \|f\|_\infty.
$$
Recall that $\|f\|_\infty \leq |f|_v$. Summarizing the estimates above, we finally obtain\\
$\DS 
\|P_nf - \brP f\|_1\leq L'\delta_0|f|_v.
$
\qed
 
\subsection{Quenched drift and variance.}
For $k<m$, let 
$$ \fa_{m, k}=\int_\mathbb{T}[  r_m(G_{m-1} \circ \dots \circ G_{m-k} (x))dx].$$ 
For $k=m-1$ we set $\fa_m=\int_\T r_m(G_{1,{m-1}} x)dx$.

 The properties of $\fa_m$ are summarized below.

\begin{lemma}
\label{LmQExp}
There are constants $C_1, C_2, C_3 >0 $ $0< \theta_1, \theta_2, \theta_3<1$ such that

(a) For each $k$ the sequence 
$m\to \fa_{m,k}$ is stationary and $\left| \fa_{m, k}-\fa_m\right|<C_1 \theta_1^k.$

(b) There exists the limit $\DS \ba=\lim_{m\to\infty} \bE(\fa_m)$ and moreover 
$\DS \left|\bE(\fa_m)-\ba\right|\leq C_2 \theta_2^m.$

(c) {\rm Cov}$(\fa_{n_1}, \fa_{n_2})\leq C_3 \theta_3^{|n_2-n_1|}. $

(d) There exists $D^2\geq 0$ such that 
$\DS \frac{\left[\int_\T \tau_m(x) dx\right]-m \ba}{\sqrt m} \Rightarrow \cN(0, D^2)$ as $m\to\infty.$

(e) For each $\eps>0$ there exists $C(\omega)$ such that for each 
$n_1, n_2<10 N$ such that $|n_2-n_1|\leq N^{3/4}$ we have
$$\Big|\int_\mathbb{T}\tau_{n_2}dx-\int_\mathbb{T}\tau_{n_1}dx-\ba (n_2-n_1)\Big|\leq C(\omega) N^{3/8+\eps}. $$
\end{lemma}

\begin{remark}
{Note that we do not claim that $D$ in part (d) is not equal to zero.}
\end{remark}

\begin{proof}
(a) Note that for $m>k$
$$
\Big|\int_\T f_m(G_{m-1} \circ \dots \circ G_{m-k} x) dx-
\int_\T f_m(G_{m-1} \circ \dots \circ G_{1} x) dx \Big|
$$
$$
\leq \int_\T  f_m |P_m\dots P_{m-k}[ \mathbf{1} -P_{m-k-1}\dots P_{1} \mathbf{1}]|dx 
\leq C_1 \theta^k,
$$
where the last estimate is due to the exponential mixing condition \eqref{dec}.

(b) By part (a)
\begin{equation}\label{cau}
\DS \bE(\fa_m)=\bE(\fa_{m,k})+O(\theta_1^k)=\bE(\fa_{k,0})+O(\theta_1^k).
\end{equation} 
Hence, $|\bE(\fa_n)- \bE(\fa_m)|< C \theta^k_1$, for $n>m>k$, which shows that the sequence $\{\bE(\fa_n)\}_{n \geq 1}$ is a Cauchy sequence. Thus, we have the limit $\DS \ba=\lim_{m\to\infty} \bE(\fa_m)$. Next, by letting $m \rightarrow \infty$ in \eqref{cau} we get (b).

(c) Assume $n_1>n_2$. By (b) we can write $|\fa_{n_1}- \fa_{n_1, n_1-n_2}| \leq C_1\theta^{|n_2-n_1|}$. Hence
$$
\bE[(\fa_{n_1}-\bE[\fa_{n_1}]) (\fa_{n_2}-\bE[ \fa_{n_2}])]= \bE[(\fa_{n_1,n_1-n_2}-\bE[\fa_{n_1}]) (\fa_{n_2}-\bE[ \fa_{n_2}])]+C' \theta^{|n_1-n_2|}
$$
$$
=C' \theta^{|n_1-n_2|}
$$
where the last equality is due to that fact that $\fa_{n_1,n_1-n_2}$ and $\fa_{n_2}$ are independent random variables and  $\bE[ (\fa_{n_2}-\bE[ \fa_{n_2}])]=0$.

(d) follows from (c), see \cite[Chapter XVIII]{I-L}.

 (e) also follows from (c) as is shown in \cite{Ga82}.
\end{proof}
We also need the following result
\begin{lemma}[\cite{K98}]\label{vrnce}
There is a constant $\bsigma$ such that
$\DS \lim_{n\to\infty} \frac{\var(\tau_n)}{n}=\bsigma^2$ with probability~1.
\end{lemma}

\subsection{Proof of Theorem \ref{main} for Model A}

Define $\cS(n)=\EXP(\tau_n)$.
This function is monotone so we consider
an inverse function
\begin{equation} 
\label{DefNScale}
\cZ(s)\!=\!\max(x\!: \cS(n)\!\leq\! s). 
\end{equation}
Denote 
$$\brsigma_n=\sqrt{\var(\tau_{\cZ(n)})}, \quad z_n^{*}=\max_{0\leq k\leq n}z_k, \quad b_n=\cZ(n)$$

{
\begin{lemma}
\label{LmLin}
$\cS(n)$, $b_n$, and $\sigma_n^2$ have linear growth. That is there is a constant $C$ such that
$$ \frac{1}{C} \leq \frac{\cS(n)}{n}\leq C, \quad
\frac{1}{C} \leq \frac{b_n}{n}\leq C, \quad
\frac{1}{C} \leq \frac{\brsigma_n^2}{n}\leq C.$$
\end{lemma}

\begin{proof}
Since $n\leq \tau_n\leq 3n$ we have $n\leq \cS(n)\leq 3n.$ Therefore
$n/3\leq b_n\leq n.$

The lower bound on $\brsigma_n^2$ follows from \Cref{lm-1}, see the proof of  \Cref{ThHit}.
The upper bound on $\brsigma_n^2$ follows from \eqref{dec} since

$\DS \brsigma_n^2=\sum_{n_1, n_2\leq n} \mathrm{Cov}(r_{n_1}, r_{n_2}) \leq
\sum_{n_1, n_2\leq n} C_1 \theta^{|n_2-n_1|} \leq C_2 n. $
\end{proof}
}
Let $\ba$ be as in Lemma \ref{LmQExp} and consider 
$$
P\left(\frac{z_n^*- b_n}{
(1/\ba)\brsigma_n}>t\Big)=P\Big(z_n^*>b_n + (1/\ba) t\brsigma_n\right).
$$
By definition of $z_n^*$ for every $t\in \mathbb{R}$ we have
$$
P(z_n^*>t)=P(\tau_{[t]}<n).
$$
Hence
$$
P\left(z_n^*>b_n + u\brsigma_n\right)=P(\tau_{[b_n + u\brsigma_n ]}<n)
=
P\left(\frac{\tau_{[b_n + u\brsigma_n]}- \EXP[\tau_{[b_n +u\brsigma_n]}]}{\sqrt{\var[\tau_{[b_n + u\brsigma_n ]}]}}>\frac{n- \EXP[\tau_{[b_n + u\brsigma_n ]}]}{\sqrt{\var[\tau_{[b_n + u\brsigma_n ]}]}}\right)
$$
where 
$$u=t/\ba.$$
{We claim that 
\begin{equation}
\label{RHSRen}
\lim_{n\to\infty} \frac{n- \EXP[\tau_{[b_n + u\brsigma_n ]}]}{\sqrt{\var[\tau_{[b_n + u\brsigma_n ]}]}}=t. 
\end{equation}
Indeed using \Cref{LmQExp}(e) and the linear growth on $b_n$ (Lemma \ref{LmLin}) we obtain
$$
\EXP[\tau_{[b_n + u\brsigma_n]}]=n + [\ba u \brsigma_n ] + O(n^{0.4})=n + [t\brsigma_n ] + O(n^{0.4}).
$$
Hence the numerator of \eqref{RHSRen} 
is asymptotic to $t\brsigma_n.$

To analyze the denominator denote $\DS \tau_{m,n}=\sum_{k=m, n-1} r_k$ for $m<n.$ 
Then 
$$ \var(\tau_n)=\var(\tau_m)+\var(\tau_{m,n})+\text{Cov}(\tau_m, \tau_{m,n})=
\var(\tau_m)+\var(\tau_{m,n})+O(1). $$
By the linear growth of variance we obtain
$\DS \var(\tau_n)=\var(\tau_m)+O(|n-m|).$
It follows that the denominator of \eqref{RHSRen} is 
$\sqrt{\brsigma_n^2+O(\brsigma_n)}=\brsigma_n+O(1).$

Combining the estimates for the numerator and denominator we obtain \eqref{RHSRen}.

}

Combining \eqref{RHSRen} and \Cref{ThHit} we arrive at
\begin{equation}
\label{XCLT1}
\lim_{n \rightarrow \infty}P\Big(\frac{z_n^*- b_n}{
(1/\ba)\brsigma_n}>t\Big)=\int_{-\infty}^t \frac{1}{\sqrt{2\pi}} e^{-u^2/2} du.
\end{equation}
From the definition \eqref{DefNScale} it follows that
$\EXP(\tau_{\cZ(n)})/n \rightarrow 1$, as  $n \rightarrow \infty$. Then consider
$\EXP\Big(\frac{\tau_{\cZ(n)}}{\cZ(n)}\Big) \frac{\cZ(n)}{n}$. 
By Lemma \ref{LmQExp} 
for almost all environments $\DS \lim_{n \rightarrow \infty}\EXP\left(\frac{\tau_{\cZ(n)}}{\cZ(n)}\right)=\ba$. 
Hence for almost all environments we also have 
$$
\lim_{n \to \infty} \frac{\cZ(n)}{n}=\left( \lim_{n\to\infty} \frac{\cZ(n)}{\tau(\cZ(n))}\right)
\left(\lim_{n\to\infty} \frac{\tau(\cZ(n))}{n}\right)=
\frac{1}{\ba}.
$$
Thus
$$
\lim_{n \rightarrow \infty}\frac{\brsigma_n^2}{n}=\lim_{n \rightarrow \infty}\frac{\var[\tau_{\cZ(n)}]}{n}=\lim_{n \rightarrow \infty} \frac{\var[\tau_{\cZ(n)}]}{\cZ(n)}\frac{\cZ(n)}{n}=\frac{\bsigma^2}{\ba}.
$$
Therefore \eqref{XCLT1} can be rewritten as
$$
\lim_{n \rightarrow \infty}P\Big(\frac{z_n^*- b_n}{
(1/\ba^{3/2})\bsigma\sqrt{n}}>t\Big)=\int_{-\infty}^t \frac{1}{\sqrt{2\pi}} e^{-u^2/2} du.
$$
Splitting
$$
\frac{z_n-b_n}{\sqrt{n}}=\frac{z_n - z_n^*}{\sqrt{n}}+\frac{z_n^*-b_n}{\sqrt{n}}
$$
and using that 
\begin{equation}
\label{Z-Z*}
z_n^*-1\leq z_n\leq z_n^*
\end{equation}
we obtain part (a). 

(b) We write
$$
\frac{\tau_{n}-n \ba}{\sqrt{n}}=\frac{\tau_{n}-\EXP[\tau_{n}]}{ \sqrt{n}}+\frac{\EXP[\tau_{n}] - n\ba}{\sqrt{n}}
$$
By part (a), the first term is asymptotically normal.  By Lemma \ref{LmQExp}
the second term is also asymptotically normal. Moreover, those terms are asymptotically
independent since the second term depends only on the environment, while the distribution of the
first term is asymptotically independent of the environment due to part~(a). Since the sum of two
independent normal random variables is normal, $\DS \frac{\tau_{n}-n \ba}{\sqrt{n}}$
 is asymptotically normal with zero mean and variance
 $$\sigma^2=\bsigma^2+D^2$$ where $\bsigma$ is from Theorem \ref{main}(a) and $D$ is
from Lemma \ref{LmQExp}(d).

Let $v=1/\ba$. Then for $x(n,t):=\left\lceil n v+v^{3 / 2} \sigma \sqrt{n} t\right\rceil$ we have
\begin{equation}
\mathbb{P}\left(\frac{z_{n}^{*}-n v}{v^{3 / 2} \sigma \sqrt{n}}<t\right)=\mathbb{P}\left(\frac{\tau_{x(n, t)}-x(n, t) / v}{\sigma \sqrt{x(n, t)}}>\frac{n-x(n, t) / v}{\sigma \sqrt{x(n, t)}}\right).
\end{equation}
It follows from the above definition of  $x(n, t)$ that
$$
\lim _{n \rightarrow \infty} \frac{n-x(n, t)1/v}{\sigma \sqrt{x(n, t)}}=-t
$$
Hence the CLT for $z_n^*$ in the annealed case follows from the discussion above.
Using \eqref{Z-Z*} we obtain the annealed CLT for $z_n.$
\qed

\section{Auxiliary results for Model B}
We now turn to the proof of Theorem \ref{ThHit} (b). Due to the complexity of the dynamics the maps $G_n$ will be more complicated. 
In particular, the walker can make arbitrary large number of backward steps before moving from
site $n$ to site $n+1.$
Thus for Model B the maps $G_n$ have infinitely many branches and the first hitting times maps $\tau_n$ are unbounded. So we can no longer apply Theorem \ref{CRThm}. 
 In the rest of the paper we establish properties which help us to verify that the transformations $\{G_n\}$ and the maps $\{\tau_n\}_{n\geq 1}$ satisfy 
the conditions of Theorem \ref{pr-1} 
{ which extends Theorem~\ref{CRThm} to the case of unbounded observables.}

\subsection{Long itineraries}\label{inf-trj}
For each $n \geq 1$ and $t \in D_n^{(2)}$ we consider the set of all $x \in \mathbb{T}$ for which the walker, starting its journey from $(x,m)$, subsequently visits the sites $m+s_1(t), m+s_2(t), \dots,m + s_{n}(t)$, i.e. for all $0\leq k \leq n$ we have
$$
\pi_\mathbb{Z}(F^k(m,x))=m + \sum_{\ell=0}^k s_{\ell}(t),
$$
where $F^0(m,x)=(m,x)$. We denote the set of all such $x$ by $A_{t,n,m}$. 
Observe that, for large values of $a$ this set is not empty for arbitrary $m,t$. Indeed, 
for large values of $a$ any point $(z,x)$ will have a preimage inside any of the four intervals $(z-1, W_{z-1})$,  $(z-1, W_{z-1}^c)$, $(z+1, W_{z+1})$ and  $(z+1, W_{z+1}^c)$. This means that at each step, by choosing the backward image in an appropriate way, we can make the walker travel  in an arbitrary prescribed way. This proves that any trajectory is possible in Model~B.

Note that 
$A_{t,n,m}$ can be written in the following way. For $n=1$
$\DS 
A_{t,n,m}=W_{m,t_1},
$
and for $n>1$
$$
x\in W_{s_0(t) + m,t_1},$$
\begin{equation} T_{s_0(t)+m}(x)\in W_{s_1(t) + m,t_2},\label{mk}\end{equation}
$$\cdots$$ 
$$T_{s_{n-2}(t) + m}\circ T_{s_{n-3}(t) + m}\circ \cdots \circ T_{m	}(x)\in W_{s_{n-1}(t) + m,t_{n}} .$$

If $G_m$ in \eqref{induce} is well defined for almost all $x \in \mathbb{T}$, then one can see that
$$
G_m(x)=\sum_{n=1}^\infty\sum_{t \in R_{n}^{(2)}} T_{s_{n-1}(t)+m}\circ T_{s_{n-1}(t)+m}\circ \cdots \circ T_{s_0(t) + m}(x)\mathbbm{1}_{A_{t,n,m}}(x).
$$

Note that in \eqref{mk} the maps $T_{s_k(t)}$ and the gates $W_{s_k(t),t_{k+1}}$ are perturbations of the map $\brT$ and the gates $\brW_{\pm}$ (depending on the sign of $t_{k+1}$), so
due to the general nature of our argument, we will assume that the walker starts at 0, that is $m=0$, and replace the maps $T_{s_k(t)}$ and the gates $W_{s_k(t),t_{k+1}}$ with  maps $T_k(x)=(ax+h_k(x))\Mod 1$ and gates $W_{k,t_{k+1}}$ respectively. Thus the index $k$ in $W_{k,t_{k+1}}$ no longer represents the position of the walker, but rather is a numbering parameter. However, we leave the sequence $\{t_k\}_{k=1}^n$ the same. Hence, for all $n \geq 1$ and $t \in D^2_n$ we can write
\begin{equation}
\label{DefANT}
A_{t,n}=\bigcap_{\ell=0}^{n-1} (T_{1,{\ell}})^{-1}(W_{\ell, t_{\ell+1}}),
\end{equation}
where $T_1^0(x)=x$.

Since the sets $(T_{1,{\ell}})^{-1}(W_{\ell, t_{\ell+1}})$ consist of finitely many disjoint intervals, 
one can check that the set $A_{t,n}$ satisfies the conditions of Definition \ref{I-def}. 
Thus the collection $\I(A_{t,n})$ is well defined.

\begin{lemma}\label{cov-lab}
Let $\ell_1<\ell_2$ and let $\mathcal{P}_{\ell_1}$ and $\mathcal{P}_{\ell_2}$ be Markov partitions of $T_{1,\ell_1}$ and $T_{1,\ell_2}$ respectively (see \Cref{markov}). Then any interval from $\mathcal{P}_{\ell_2}$ can intersect at most two intervals from $\mathcal{P}_{\ell_1}$. 
\end{lemma}
\begin{proof}
Let $\mathcal{P}_{\ell_1}$ and $\mathcal{P}_{\ell_2}$ be Markov partitions of $T_{1,\ell_1}$ and $T_{1,\ell_2}$ respectively. We need to show that any $p \in \mathcal{P}_{\ell_2}$ can intersect at most two intervals from $\mathcal{P}_{\ell_1}$.
If some $p \in \mathcal{P}_{\ell_2}$ intersects more than two elements from $\mathcal{P}_{\ell_2}$ there should exist an element $p'\in \mathcal{P}_{\ell_2}$ so that $p' \subset p$.
Note that, by definition
$
T_{1,{\ell_2}}(p)= \mathbb{T}.$ However
\begin{equation}\label{int1}
T_{1,{\ell_2-1}}(p)\neq \mathbb{T}.
\end{equation}
But since $p' \subset p$ and $\ell_1<\ell_2$, then clearly
$\mathbb{T}=T_{1,{\ell_1}}(p')\subset T_{1,{\ell_2-1}}(p)$
contradicting \eqref{int1}.
\end{proof}

The next propositions helps us to understand the long trajectories. Although it may not be visible at first glance, the argument is close to and was inspired by the classical growth lemma that is fundamental in the study of billiard maps. 

\begin{proposition}\label{unbdd}
Let $t\in D^{(2)}_n$ and $\ell \geq 0$ be the number of $-1$s in $t$. Then, for all $a\in \mathbb{N}$ sufficiently large, there exists $k=k(a)\in \mathbb{N}$ and $\bar{\delta}(a)>0$ so that if $\delta\leq \bar{\delta}(a)$ then there is a 
Markov partition $\mathcal{P}_{n+1}$ of $T_{1,{n+1}}$ so that the number of 
Markov partition elements required to cover the set $A_{t,n}$
can be estimated as follows
\begin{equation}\label{part_in}
\{p\in \mathcal{P}_{n+1}:p \cap A_{t,n}\neq \emptyset\}\leq (k+1)^{l}(a-k+4)^{n-l}.
\end{equation}
Furthermore, each partition element contains at most one interval from $A_{t,n}$, i.e. for any $p \in \mathcal{P}_n$
\begin{equation}\label{one-int}
\#\{I\in \I(A_{t,n}): I\cap p \neq \emptyset \}\leq 1,
\end{equation}
and the following two bounds take place
\begin{equation}\label{part}
a|\brW_{-1}|\leq k \leq a|\brW_{-1}| + 2,
\end{equation}
and
\begin{equation}\label{meas}
|A_{t,n}|\leq \frac{(k+1)^{l}(a-k+4)^{n-l}}{(a-\bar\delta)^{n}}.
\end{equation}

\end{proposition}
\begin{proof}
The Markov partition of $T_\ell$ (see \eqref{markov}) will be denoted by
\begin{equation}\label{prt-mrk}
\mathbb{P}_\ell=\{P_{i}^\ell\}_{i=0}^{a-1}, \quad \ell \geq 1.
\end{equation}
To make it unique we assume that for all $\ell \geq 1$, 
$0\leq i \leq a-1$
 we have $T_\ell(\partial P_{i}^\ell)=\{0,1\}$. The Markov partition of $\brT$  will be denoted by $\{P_{i}\}_{i=0}^{a-1}$.

We now take $a$ so large and $\bar{\delta}(a)$ so small that for any $\ell \geq 1$, there is $0\leq i,j \leq a-1$ such that
$$
P_{i}^\ell \subset \interior W_{\ell,1}\quad
P_{j}^\ell \subset \interior W_{\ell,-1},
$$
where $\interior$ stands for the interior of the set.
Let $C_1$  and $C_{-1}$ be the minimal subsets of indexes $\{0, \dots, a-1\}$ such that for all $\ell,k \geq 0$
\begin{equation}\label{cov-l1}
\overline W_{k,1} \subset \interior\Big(\bigcup_{i \in C_1}P_i^\ell\Big) \quad \overline W_{k,-1} \subset \interior\Big(\bigcup_{i \in C_{-1}}P_i^\ell\Big).
\end{equation}

One can check that
\begin{equation}
\label{Inter4}
\# \{C_1 \cap C_{-1}\}\leq 4.
\end{equation}
Note that we have equality in \eqref{Inter4} iff $\overline W_{-1}= \overline{\Big(\cup_{i \in C_{-1}}P_i\Big)}$, where $P_i$ are from the Markov partition of $\brT$ defined in \eqref{prt-mrk}.
Therefore, if 
\begin{equation}\label{KOA}
k(a)=k= \# C_{-1} 
\end{equation}
then
\begin{equation}\label{part-est1}
\# C_{1} \leq a - (k - 4).
\end{equation}
One can also check that
\begin{equation}\label{lm-bd}
|\brW_{-1}|\leq \frac{k}{a}\leq |\brW_{-1}|+\frac{2}{a}.
\end{equation}
We now consider the set of all $x \in \mathbb{T}$ for which
$$
x\in { \bigcup_{i \in C_{t_1}} P_i^1}, 
\quad
T_{1}(x)\in { \bigcup_{i \in C_{t_2}} P_i^2}, 
\cdots,
T_{1,{n-1}}(x)=T_{n-1}(T_{n-1}( \cdots (T_1(x))\in { \bigcup_{i \in C_{t_n}} P_i^n}
$$
where $t=(t_1, \dots, t_n)$. Set
$$
B_n=\bigcap_{k=0}^{n-1} T_{1,k}^{-1}\left(\bigcup_{ i\in C_{t_{k+1}}} P_i^{k+1}\right),
$$
where $T_{1,0}(x)=x$. By \eqref{cov-l1} we have that $A_{t,n}\subset B_n$. Observe that for $1<k<n-1$ the collections $\mathcal{P}_{k}=\{T_{1,k}^{-1}(P_{i}^{k+1}):i=0,\dots,a-1\}$ (we have $k$ preimages and each preimage is counted separately) constitutes a Markov partition for $T_{1,k+1}$. Indeed, for any $p \in \mathcal{P}_k$
$$
T_{1,k+1}(p)=T_{k+1}(T_{1,k}(p))=T_{k+1}(P_{i}^{k+1})=\mathbb{T}.
$$

By assumption for each $0\leq k \leq n$ there exists $P_i^k, P_j^k$ so that $P_i^k \subset W_{k,-1}$ and also $P_j^k \subset W_{k,1}$. This means that for $t=\pm 1$, $T_{1,k}^{-1}(W_{k,t})$ contains an element from $\mathcal{P}_k$ in its interior, i.e. a Markov partition element from $T_{1,k+1}$.
We now bound the number of partition elements from $\mathcal{P}_{n-1}$ that cover the set $B_n$. For this observe that $B_n$ is the union of all partition elements $p$ from $\mathcal{P}_{n-1}$ for which we have that
for all $k=0,\dots, n$
\begin{equation}\label{cond-part}
T_{1,k}(p)\subset P_i^{k+1}, \text{for some }i \in C_{t_k}.
\end{equation}
Note that in term of the indexes $i$, this is the set of elements such that
$$
\{i_1, \dots, i_n: i_k \in C_{t_k}, 0\leq k \leq n-1\}.
$$
The cardinality of this set can be bounded from above as follows
$$
{(\# C_{-1}})^\ell (\# C_{1})^{n-\ell}
\leq k^\ell (a-(k-4))^{n-\ell}.
$$
Next, by the mean value theorem if $p=[a,b]$ then
$$
1=T_{1,{n}}(a)-T_{1,{n}}(b)=(T_{1,n})'(\zeta)(a-b).
$$
Thus
$$
1 = |(T_{1,n})'(\zeta)||p|.
$$
Hence
$$
|p|=\frac{1}{|(T_{1,n})'(\zeta)|}\leq \frac{1}{|(a+h_n'(T_{1,{n-1}}(\zeta)))\cdots (a+h_1'(T_1(\zeta)))|}\leq \frac{1}{(a-\bar\delta)^n}.
$$
This implies
$$
|A_{t,n}|\leq \frac{(k+1)^{l}(a-k+3)^{n-l}}{(a-\delta)^{n}}.
$$

Next note that the condition \eqref{part} follows from \eqref{lm-bd}.

We now show \eqref{one-int}. For this we observe that each interval $I \in \I(A_{t,n})$ contains a Markov partition element from $\mathcal{P}_{n-1}$ in its interior. Indeed,
assume the opposite. Then there exists two  intervals $I_1, I_2 \in \mathcal{I}(A_{n,t})$, so that
for some $p \in \mathcal{P}_{n-1}$ we have
$$
I_1 \cap p\neq \emptyset \text{ and } I_2 \cap p\neq \emptyset.
$$ 
Then there exists $\ell < n$ and two disjoint intervals $J_1, J_2 \in \mathcal{I}((T_{1,\ell})^{-1}(W_{\ell,t_{\ell+1}}))$, so that
$$
I_1 \subset J_1 \hbox{ and }I_2 \subset J_2.
$$
Then clearly $\dist(J_1, J_2)< |p|$, since one of their endpoints belongs to $p$.
Next, by construction there should be a Markov partition element $p' \in \mathcal{P}_{\ell}$ such that $p'$ is between $J_1$ and $J_2$. This means that $p' \subset p$. However this is not possible, since $p$ will also intersect the partition elements that are neighboring with $p'$. But this means that $p$ intersect at least $3$ partition element from $\mathcal{P}_\ell$ which is not possible due to Lemma~\ref{cov-lab}. This proves \eqref{one-int}.
\end{proof}

\begin{proposition}\label{drift}
Assume $|W_{-1}|<1/2$. Then there exists $\rho<1$ and $C>0$ such that for every $t=(t_1,\dots, t_n) \in D_n^{(2)}$, with $t_1+\dots + t_n\leq 1$, we have
\begin{equation}\label{drift-1}
|A_{t,n}|\leq C \frac{\rho^{n}}{2^n}.
\end{equation}
\end{proposition}
\begin{proof}
Let $l$ be the number of times the walker goes left during his journey. By \eqref{meas}
$$
|A_{t,n}|\leq\frac{(k+1)^{l}(a-k+4)^{n-l}}{(a-\delta)^{n}}.
$$
Next, due to $l\geq \frac{n-1}2$ we have that
$$
\frac{(k+1)^{l}(a-k+4)^{n-l}}{(a-\delta)^{n}}\leq C
\frac{(k+1)^{\frac{n}{2}}(a-k+4)^{{\frac{n}{2}}}}{(a-\delta)^{n}}
$$
where we have used that $k<a-k+4$ for large $a$, which holds since $|W_{-1}|<|W_1|.$
Rewrite
$$
2^n\frac{(k+1)^{\frac{n}{2}}(a-k+4)^{{\frac{n}{2}}}}{(a-\delta)^{n}}= 
\left( \frac{4(k+1)^{1/2}(a-k+4)^{1/2}}{(a-\delta)}
\right)^n.
$$
Next
$$
\lim_{a \rightarrow \infty, \delta \rightarrow 0} \frac{4(k+1)^{1/2}(a-k+4)^{1/2}}{(a-\delta)}=4|W_{-1}|(1-|W_{-1}|)<1
$$
where the last inequality is due to $|W_{-1}|<\frac{1}{2}$. 

Thus \eqref{drift-1} holds with
$\DS \rho = \frac{4(k+1)^{1/2}(a-k+4)^{1/2}}{(a-\delta)}.$
\end{proof}

\begin{lemma}\label{existence}
If $|\brW_{-1}|<1/2$, then for $a\in \mathbb{N}$ sufficiently large and $\delta=\delta(a)$ small, the maps $\{G_n\}_{n \in \mathbb{Z}}$ are 
defined almost everywhere. 
\end{lemma}
\begin{proof}
We need to show that
$$
\mes\left(x: z_n(x)
\leq 0, \forall n \geq 0\right)=0.
$$
Since the number of all trajectories of length $n$ is equal to $2^n$, then in view of Proposition \ref{drift} we can write
\begin{equation}
\label{LongExcursion}
\mes\left(x: z_n(x)
\leq 0\right)\leq 2^n \frac{\rho^n}{2^n}\leq \rho^n.
\end{equation}
This completes the proof.
\end{proof}

\subsection{Short itineraries}\label{shrt-it}

For $t \in D^2_n$ and gates $\{W_{k, t_{k+1}}\}_{k=0}^n$. Recall \eqref{DefANT}.
Let
\begin{equation}\label{hom0}
H_{t,n}=\bigcap_{k=0}^{n-1} (\brT^{k})^{-1}(\brW_{ t_{k+1}}),
\end{equation}
where $\brT=ax \Mod 1$.

\begin{proposition}\label{short-it}
For given $n \geq 1$ and $t \in D_n^{(2)}$, we have that

(a)
$\DS 
\#\I(A_{t,n})\leq 4a^{n-1}.
$ \smallskip

$(b)$
For every $\varepsilon>0$ there exists $\bar\delta(a)$ such that if $\delta\leq \bar\delta=\bar\delta(a)$ for every closed interval $I$ in the collection $\I(H_{t,n})$ there is 
a unique $J \in \I(A_{t,n})$, so that $|I \triangle J|<\varepsilon$. Moreover, $\bar\delta$ can be taken so small that if $\delta<\bar\delta$ then for every $x \in I \cap J$ we have
\begin{equation}\label{uni-close}
|T_{1,n}(x)-\brT^n(x)|\leq \varepsilon,
\quad \left|\frac{1}{(T_{1,n}(x))'} -\frac{1}{(\brT^n(x))'}\right|\leq \varepsilon,
\end{equation}
and
$$
|A_{t,n} \triangle H_{t,n}|\leq \varepsilon.
$$
\end{proposition}
\begin{proof}

It follows from the definition that the boundary of $A_{t,n}$ consists of 
$$
\partial A_{t,n}\subset\bigcup_{k=0}^{n-1} (T_{1,{k}})^{-1}(\partial W_{k, t_{k+1}})).
$$
In the same way
$$
\partial H_{t,n}\subset\bigcup_{k=0}^{n-1}(\brT^{k})^{-1}( \partial\brW_{t_{k+1}}).
$$
Note that $\#\partial(T_1^{k})^{-1}(W_{k, t_{k+1}})\leq 2 a^{k}$. Hence
$$
\#\partial A_{t,n}\leq  2 + 2a + \dots + 2 a^{n-1}=2 \frac{a^n-1}{a-1}\leq 4a^{n-1},
$$
for $a \geq 2$. This proofs part (a). \\

Obviously, for arbitrary $\varepsilon>0$ we can take $\bar\delta$ so small that $|T_{1,n}(x)-\brT^n(x)|<\varepsilon$ for all $x \in \mathbb{T}$ (the second statement in \eqref{uni-close} is also similar). Clearly the sets $\partial((T_{1,{k}})^{-1}(W_{k, t_{k+1}}))$ and $\partial(\brT^{k})^{-1}(\brW_{t_{k+1}})$ will be close to each other under small perturbations and for given $k$. Hence, for each $I \in \I(H_{t,n})$ its endpoints will change a little under small perturbations of $\brT$. Thus we obtain that for some $J \in \I(A_{t,n})$ we have	 $|I \triangle J|<\varepsilon$. 
Observe, 
that $\I(A_{t,n})$ may contain other intervals too, that come into existence under small perturbations of the maps $\brT^k$. However these intervals occupy a set of small measure.
Note that for any $k \geq 0$
$$
A_{t,n} \triangle H_{t,n}\subset((T_{1,{k}})^{-1}(W_{k, t_{k+1}}))\triangle ((\brT^{k})^{-1}(\brW_{t_{k+1}})).
$$
Thus for $\brdelta$ sufficiently small 
$|A_n \triangle A|<\varepsilon$.
\end{proof}

\begin{lemma}\label{second}
If $a$ is sufficiently large and $\delta(a)$ is small enough then
$$
\left|\left(\frac{1}{(T_{1,n}(x))'}\right)'\right|=\left|\frac{(T_{1,n}(x))''}{(T_{1,n}(x)')^2}\right|< D,
$$
where $\DS D:=\sup_{k\geq 1}\sup_{x \in \mathbb{T}} |h_k''(x)| <\infty$.
\end{lemma}
\begin{proof}
We have
$$
(T_{1,n}(x))'=T'_n(T_{1,{n-1}}(x))T'_{n-1}(T_{1,{n-2}}(x))\cdots T_1'(x).
$$
Hence
$$
(T_{1,n}(x))''=\sum_{k=1}^n T'_n(T_{1,{n-1}}(x))\cdots T''_{k}(T_{1,{k-1}}(x))(T_{1,{k-1}}(x)')^2.
$$
We rewrite this as follows
$$
(T_{1,n}(x))''=\sum_{k=1}^n \frac{(T_{1,n}(x))'}{T_k'(T_{1,{k-1}})}T_k''(T_{1,{k-1}})T_{1,{k-1}}(x)'.
$$
Thus
$$
|(T_{1,n}(x))''|\leq \sum_{k=1}^n \frac{|T_{1,n}(x)'|}{ |a-\delta|}D|T_{1,{k-1}}(x)'|.
$$
Hence
$$
\left|\frac{(T_{1,n}(x))''}{(T_{1,n}(x)')^2}\right|\leq \sum_{k=1}^n \frac{D}{ (a-\bar\delta)}\frac{|T_{1,{k-1}}(x)'|}{|T_{1,n}(x)'|}\leq \sum_{k=1}^n \frac{D}{(a-\bar\delta)^{n-k+1}}
\leq \sum_{m=1}^\infty \frac{D}{(a-\bar\delta)^m}=\frac{D}{a+1-\brdelta}
$$
The last sum is smaller than $D$ if $\brdelta<1$. 
\end{proof}

\section{Properties of transfer operators}\label{LY-sct}

\begin{proposition}\label{prop-ly}
Let $a$ be as in Lemma \ref{existence}.
Then there is $\bar\delta(a)$, such that if $\delta<\bar\delta(a)$, then the collection $\{P_n\}_{n \geq 1}$  satisfies property \eqref{LY}:
there exists a constant $C(a, \bar\delta)>0$ such that for every $n \geq 1$
\begin{equation}
\label{Var3A}
\mathrm{V}\left(P_n f\right) \leq {\frac{3}{4}} \mathrm{V} (f)+C\|f\|_{1},
\end{equation}
\end{proposition}

\begin{proof}
Throughout the proof, instead of the notations $G_n$ and $P_n$, we will use the generic notations $G$ and $P$. As earlier, for $A_{t,n}$ let $\I(A_{t,n})$ be the set defined in \ref{I-def}.

As in the proof of Lemma \ref{lz-prf}, {
$\DS 
V(P f)\leq 
I+\RmII$ where $I$ and $\RmII$ are given by \eqref{VProd}.}

Similarly to the proof of \eqref{lz-prf} 
\begin{equation}\label{est-I}
\text{I} \leq \frac{V(f)}{a-\brdelta}+
L \|f\|_1\quad \mathrm{where}\quad L=\sup_{x\in \mathbb{T}} \frac{|G''(x)|}{|(G'(x))^2|}.
\end{equation}
By Proposition \ref{second}, 
$L\leq D.$

Note that for $x \in A_{t,n}$ we have $(a-\bar\delta)^n\leq|G'(x)|\leq (a+\bar\delta)^n$. 
For large values of $n$ we estimate $\RmII$ using \eqref{1G-3} taking $I=\mathbb{T}$. Then for each interval $I_j \in \I(A_{t,n})$ we have
$$
\Big|\left(\frac{f}{G^{\prime}}\right)\left(\sigma_{i} \alpha_{i}\right)\Big|+\Big|\left(\frac{f}{G^{\prime}}\right)\left(\sigma_{i} \beta_{i}\right)\Big| \leq \frac{1}{(a-\bar\delta)^{n}} V_{\mathbb{T}}(f) + \frac{2}{(a-\bar\delta)^{n}}\int_{\mathbb{T}} |f|dx
$$
\begin{equation}\label{ly-2}
\leq \frac{2}{(a-\bar\delta)^{n}}(V(f)+\|f\|_1)=\frac{2}{(a-\bar\delta)^{n}} |f|_{BV}.
\end{equation}
Hence, by Proposition \ref{drift}
$$
\sum_{n=N}^\infty\sum_{t\in R_n^{(2)}}\sum_{I \in \I(A_{t,n})}\Big|\left(\frac{f}{G^{\prime}}\right)\left(\sigma_{i} \alpha_{i}\right)\Big|+\Big|\left(\frac{f}{G^{\prime}}\right)\left(\sigma_{i} \beta_{i}\right)\Big|
$$
\begin{equation}\label{lrg-1}
\leq \sum_{n=N}^\infty\frac{\#\I(A_{t, n})}{(a-\delta)^{n+1}}|f|_{BV}\leq \sum_{n=N}^\infty\rho^n |f|_{BV}.
\end{equation}
We now consider the terms $n<N$. First, for $I\in \I(A_{t,n})$ with $|I|>\frac{1}{(a-\bar\delta)^{n}}$ we repeat the 
argument from the proof of Lemma \ref{lz-prf} and divide $I$ into intervals $\{J\}$ of length $\frac{1}{2(a-\bar\delta)^{n+1}}$ and an interval $J'$ with $\frac{1}{2(a-\bar\delta)^{n+1}}<|J'|<\frac{1}{(a-\bar\delta)^{n+1}}$. Then, we use the estimate \eqref{1G-2} with $\gamma \geq (a-\brdelta)^n$ to obtain
$$
\Big|\frac{f}{G^{\prime}}\Big|(\sigma_{j} \alpha_{j})+
\Big|\frac{f}{G^{\prime}}\Big|(\sigma_{j} \beta_{j})\leq \frac{V_{J}(f)}{(a-\bar\delta)^{n}}+\frac{4 (a-\brdelta)^{n+1}}{(a-\bar\delta)^{n}}\int_{I_j}  |f(x)| dx.
$$
Summing over all such intervals $I \in \I(A_{t,n})$, for all $n \leq N$ and $t \in R_{n}^2$ we get
\begin{equation}\label{md-1}
\sum_{n \leq N}\sum_{t \in R_n^{(2)}}\sum_{J: |I|>\frac{1}{(a-\brdelta)^{n}}}\left(\Big|\frac{f}{G^{\prime}}\Big|(\sigma_{j} \alpha_{j})+
\Big|\frac{f}{G^{\prime}}\Big|(\sigma_{j} \beta_{j})\right)\leq \frac{V_{\mathbb{T}}(f)}{a-\bar\delta}+ C_N(a, \bar\delta)\|f\|_1.
\end{equation}
Now it remains to deal with intervals with lengths $|I|<\frac{1}{(a-\bar\delta)^{n}}$. To this end we again use the bound \eqref{ly-2}. Then, in view of Proposition \ref{short-it}(a)
$$
\sum_{I: |I|\leq \frac{1}{(a-\brdelta)^n}}V\Big(f\left(\sigma_{j} x\right) \frac{1}{\left|G^{\prime}\left(\sigma_{j} x\right)\right|} 1_{G\left(I_{j}\right)}\Big)\leq \frac{4a^{n-1}}{(a-\bar\delta)^{n}}\|f\|_{BV}.
$$
Summing over $n<N$
$$
\sum_{n \leq N}\sum_{I:|I|\leq \frac{1}{(a-\brdelta)^n}}V\Big(f\left(\sigma_{j} x\right) \frac{1}{\left|G^{\prime}\left(\sigma_{j} x\right)\right|} 1_{G\left(I_{j}\right)}\Big)
$$
\begin{equation}\label{sm-1}
\leq \frac{4}{(a-\bar\delta)}\sum_{n \leq N}2^n \Big(\frac{a}{a-\bar\delta}\Big)^{n-1}\|f\|_{BV}\leq 
\frac{2^{ N+3}}{(a-\bar\delta)}\|f\|_{BV}
\end{equation}
if $\bar\delta$ is small enough.

Now, taking $N$ large in \eqref{lrg-1}, the slope $a$ large in \eqref{sm-1} and taking \eqref{md-1}, \eqref{est-I} into account { we obtain
 \eqref{Var3A}}.
\end{proof}

Recall the definition of $r_n$ in \eqref{induce}. Observe now that for any $t \in R^{(2)}_n$ and $k\geq 0$ we have
\begin{equation}\label{prpt}
r_m(x)=2k+1, \hbox{ for any }x \in A_{t, 2k+1,m}.
\end{equation}
Below $P$ and $r$ will stand for a generic transfer operator $P_m$ and the first hitting times map $r_m$ for $m \in \mathbb{Z}$. As earlier we will drop the index $m$.

\begin{proposition}\label{bdd-1}
There exists $\bar\delta_0$ so that if $\delta<\bar\delta_0$ then there exists $D>0$ such that
$$
|P(f)|_{BV} \leq D|f|_{BV},
$$
and
$$
{|P(r f)|_{BV} \leq D|f|_{BV}, \quad |P(r^2 f)|_{BV} \leq D|f|_{BV}, \quad \forall f \in BV}.
$$
\end{proposition}
\begin{proof}
We follow  the argument of Proposition \ref{prop-ly}. If in \eqref{lrg-1} instead of $f$ we consider the function 
 $r f$, then using an identity $V_I(af)=aV_I(f)$, 
valid for  $a\in \R$ and an interval $I$, and in view of \eqref{prpt} we obtain
\begin{align}\label{bd-8}
V(P (r f))\leq \sum_{n=1}^\infty C_1 n\rho^{n}|f|_{BV}.
\end{align}
Thus
$$
V(P (r f))
 \leq C_2 |f|_{BV}. 
$$
Observe also that
$$
\int_\mathbb{T} |P(r f)|dx\leq \int_\mathbb{T} P(r|f|)dx =\int_\mathbb{T} r|f|dx\leq |f|_{BV} \int_\mathbb{T}r dx.
$$
Thus by \eqref{meas}
$$
\int_\mathbb{T}r dx \leq \sum_{\ell=0}^\infty \sum_{t \in R^{(2)}_{2\ell + 1}}(2\ell + 1)|A_{t, 2\ell + 1}|<C_1 \sum_{\ell=0}^\infty (2\ell + 1)\rho^{\ell}<\infty.
$$
Combining the above estimates we get
$$
|P(r f)|_{BV}=V(P(r f))+\|P(r f)\|_1\leq D |f|_{BV}.
$$
The estimates for $|P(r^2 f)|_{BV}$ and $|P(f)|_{BV}$ are similar.
\end{proof}

Let $P, r$ and $\brP, \brr$ be the transfer operator and the first hitting times map of the perturbed and unperturbed cases respectively.

\begin{proposition}\label{bdd-norm-2}
Given $\varepsilon>0$, $\delta_0(a)$ can be taken so small that if $\delta\leq \delta_0(a)$, then
\begin{equation}\label{d-dist}
\|P(f)-\brP(f)\|_1\leq \varepsilon|f|_{BV}, \quad \|P(f)-\brP(f)\|_2\leq \varepsilon|f|_{BV}\quad  f \in BV
\end{equation}
and
\begin{equation}\label{df-dist}
\|P(r f)-\brP(\brr f)\|_1\leq \varepsilon|f|_{BV}, \quad \|P(r f)-\brP(\brr f)\|_2\leq \varepsilon|f|_{BV}\quad  f \in BV.
\end{equation}
For $\delta$ small, we will also have
\begin{equation}\label{dist-ob}
\|r - \brr\|_2\leq \varepsilon.
\end{equation}
\end{proposition}
\begin{proof}
Note that if we have the first statement in \eqref{d-dist}, then
 by Proposition \ref{bdd-1}
$$
\|\brP(\brr f)-P(r f)\|_2\leq\sqrt{\|\brP(\brr f)-P(r f)\|_\infty \|\brP(\brr f)-P(r f)\|_1 }
$$
$$
\leq \sqrt{2D|f|_{BV}\int_\mathbb{T} |\brP(r f)-P(\brr f)|dx}
\leq \sqrt{2D \varepsilon }|f|_{BV}.
$$
Hence, the second statement in \eqref{d-dist} follows from the first. In a similar way we can show that the second statement in \eqref{df-dist} follows from the first. 

We will show the first estimate in \eqref{df-dist}, as the proof of \eqref{d-dist} is similar. 
We have 
$$
P(rf)(x)=\sum_{y: G(y)=x}\frac{r(y)f(y)}{\left|G^{\prime}(y)\right|}.
$$
For given $N \in \mathbb{N}$ we write
\begin{align*}
\|P(r f)-\brP(\brr f)\|_1 &\leq \|P(\mathbf{1}_{[1,N]}(r) f)-\brP(\mathbf{1}_{[1,N]}(\brr)f)\|_1\\
& + \|P(\mathbf{1}_{(N,\infty]}(r) f)\|_1 + \|\brP(\mathbf{1}_{(N,\infty]}(\brr) f)\|_1.
\end{align*}
By Lemma \ref{l2-norm}
$$
\|P(\mathbf{1}_{(N,\infty)}(r) f)\|_1= \|\mathbf{1}_{(N,\infty)}(r) f\|_1 \leq |f|_{BV} \|\mathbf{1}_{[N,\infty)}(r) \|_1.
$$
For the last expression we have {by \eqref{LongExcursion}}
$$
\|\mathbf{1}_{(N,\infty)}(r) \|_1\leq 	C\sum_{n> N}n \rho^{n}.
$$
Hence
\begin{equation}\label{larg-N}
\|\mathbf{1}_{(N,\infty)}(r) P(r f)-\mathbf{1}_{(N,\infty)}(\brr)\brP(\brr f)\|_1\leq 2C \sum_{n> N}n \rho^{n}|f|_{BV}.
\end{equation}

{Take $N$ so large that $\DS 2C \sum_{n> N}n \rho^{n}\leq \eps.$
Then} it remains to study the term
$$
\|P(\mathbf{1}_{[1,N]}(r) f)-\brP(\mathbf{1}_{[1,N]}(\brr)f)\|_1.
$$
{Recall \eqref{var-bd}}.
For each $n\in \mathbb{N}$ and $t \in R_n^{(2)}$ consider the sets
$A_{t, n}\hbox{ and }H_{t, n}$ defined in \eqref{DefANT} and \eqref{hom0}.
By Proposition \ref{short-it} we have that 
for any $I \in \I(H_{t,n})$ there exists $J \in \I(A_{t,n})$ such that for any $\varepsilon>0$, $\delta_0$ can be taken so small that if $\delta<\delta_0$, then for all $x \in I \cap J$
$$
|G(x)-\brG(x)|\leq \varepsilon.
$$
We now consider the restrictions of transfer operators $P$ and $\brP$ onto the set $I \cap J$. Analogous to the proof of {Lemma \ref{d-dst}}
we can consider a pairing of the preimages of $x$. On the set $A$ of $x$ for which there is pairing, we can write
\begin{equation}
(P f(x)-\brP f(x))\Big|_{x\in A}
=\left(\sum_{y_1\in I \cap J:G(y_1)=x} \frac{\brr f(y_1)}{\brG'(y_1)}-\sum_{y_2 \in I \cap J:\brG(y_2)=x} \frac{r f(y_2)}{G'(y_2)}\right)
\end{equation}

We have that $\gamma=(a-\delta_0)\leq G'(x)|_{A_{t,n}}.$ Then by Proposition \ref{second}, 
$$|G''(x)|\leq D |G'(x)|\leq D (a+\delta)^n=K_1(n)$$ and noting that 
$|y_1-y_2|<\frac{L \delta_0}{\gamma}$,
we can repeat the same computations as in Lemma \ref{d-dst} and obtain
\begin{equation}\label{intervals}
\int_{A}\left|\frac{\brr f(y_2)}{\brG'(y_1)}-\frac{r f(y_2)}{G'(y_2)}\right|dx\leq n\|f\|_\infty \delta_0 L\left(\frac{K_1(n)}{\gamma^3}+\frac{1}{\gamma^2}\right).
\end{equation}

If there is no pairing between the preimages of $x$, then the preimage of $x$ lies in the set $A_{t,n} \triangle H_{t,n}$. 


By Proposition {\ref{short-it}(b)} we have that $|A_{t,k}\triangle H_{t,k}|<\varepsilon$, for $k \leq N$, if $\delta\leq \delta_0$. Hence, the measure of the points in $x \in \mathbb{T}$ which have a preimage in $A_{t,k}\triangle H_{t,k}$, with $k \leq N$, can be estimated as follows
$$
|G(A_{t,k}\setminus H_{t,k})|\leq |T_{1,k}(A_{t,k}\setminus H_{t,k})| \leq (a+\delta_0)^k |A_{t,k}\setminus H_{t,k}|\leq \varepsilon (a+\delta_0)^N
$$
and respectively
$$
|\brG(H_{t,k}\setminus A_{t,k})|\leq  a^k |H_{t,k}\setminus A_{t,k}|\leq \varepsilon a^N.
$$
Since for fixed $k$ and $t$ every $x$ can have at most $a^k$ many preimages under $T_{1,k}$, $(k \leq a)$, then we can estimate the measure of the points $x$ for which there is no pairing between its preimages as follows
\begin{equation}\label{cloud}
\left\|\sum_{\{x:\exists y, \text{ such that } y\in  A_{t,k}\setminus H_{t,k}, G(y)=x\}}\frac{ r(y)f(y)}{\left|G^{\prime}(y)\right|}\right\|_1 \leq \frac{N \|f\|_\infty}{(a-\delta_0)^n} a^k (a+\delta_0)^N\varepsilon,
\end{equation}
in a similar way
\begin{equation}\label{cloud1}
\left\|\sum_{\{x:\exists y, \text{ such that } y\in  H_{t,k}\setminus A_{t,k}, \brG(y)=x\}}\frac{ \brr(y)f(y)}{\left|\brG^{\prime}(y)\right|}\right\|_1 \leq \frac{N \|f\|_\infty}{(a-\delta_0)^n} a^k a^N\varepsilon.
\end{equation}
Summing the above for all $k\leq N$, $t \in R^{(2)}_k$ and considering \eqref{larg-N} and \eqref{intervals} we arrive at the estimate
$$
\|\brP{(\brr f)}-P{(r f)}\|_1\leq C\varepsilon|f|_{BV}.
$$
Since $\varepsilon$ was arbitrary, \eqref{df-dist} follows.\\

To see \eqref{dist-ob}, note that for $x \in I\cap J$ we have that
$$
\bar\tau(x)=\tau(x)=n.
$$
Hence
\begin{align*}
\|\bar\tau-\tau\|_2 \leq & \|\mathbf{1}_{(N,\infty]}(\tau)\|_2 + \|\mathbf{1}_{(N,\infty]}(\bar\tau)\|_2
 + \|\mathbf{1}_{[1,N]}(\bar\tau)-\mathbf{1}_{[1,N]}(\tau)\|_2\\
\leq & 2C_1 \sum_{\ell=N}^\infty \ell^2\rho^{\ell }+\sum_{n=1}^N 2^n N\sqrt{|A_{t,n}\triangle H_{t,n}|}.
\end{align*}

Taking $N$ large and $\delta_0$ sufficiently small, we get \eqref{dist-ob}.
\end{proof}

\begin{proposition}\label{asmpt}
Let
$$
\tilde{r}_k=r_k(G_{1,k-1}(x))-\int_\mathbb{T} r_k(G_{1,k-1}(x))dx,
\quad
S_n(x) =\sum_{k=1}^{n-1}\tilde{r}_k.
$$
Assume that 
$\|S_n\|^2=\sigma_n^2\geq Dn$, for all $n \geq 1$. Then for arbitrary $\varepsilon>0$
\begin{equation}\label{series}
\lim_{n \rightarrow \infty}\frac{\sum_{k=1}^n \int_\mathbb{T} \tilde{r}_{k}^{2}(x) 1_{\left[\varepsilon \sigma_{n}, \infty\right)}\left(\tilde{r}_{k}^{2}(x)\right)dx
}{\sigma^2_n}=0.
\end{equation}

\end{proposition}
\begin{proof}
We have 
$$
\sup_{k\geq 1}\int_\mathbb{T} r_k(G_{k-1} \circ\cdots \circ G_{1}(x))dx\leq \sup_{k \geq 1}\Big(\|P_{1,k-1}\mathbf{1}\|_\infty  \int_\mathbb{T} r_k(x)dx\Big)=R<\infty.
$$
Note that
$$
\tilde{r}_k^2(x) \leq 2 (r_k^2(x) + R^2).
$$
Also, if $2R\leq r_k(x)$ then
$
\tilde{r}_k^2(x)\geq (r_k(x)-R)^2 \geq \DS 
\frac{r_k^2(x)}{4} .
$
Since $\sigma_n$ tends to infinity as $n \rightarrow \infty$, then
$$
\int_\mathbb{T} \tilde{r}_{k}^{2}(x)1_{\left[\varepsilon \sigma_{n}, \infty\right)}(\tilde{r}_{k}^{2}(x))dx\leq 2\int_\mathbb{T} \Big(r_{k}^{2}(x)+R^2\Big)1_{\left[4\sqrt{\varepsilon \sigma_{n}}, \infty\right)}({r}_{k}^{2}(x))dx
$$
$$
\leq  2\sum_{k=[2\sqrt{\varepsilon \sigma_{n}}]}^\infty\sum_{t \in R_k^{(2)}}(k^2+R^2)|A_{t,k}|.
$$
For large $k$ we can write $k^2 + R^2 \leq 2k^2$. Then
$$
2\sum_{k=[2\sqrt{\varepsilon \sigma_{n}}]}\sum_{t \in R_k^{(2)}}(k^2+R^2)|A_{t,k}|\leq C \sum_{k=[\sqrt{2\varepsilon \sigma_{n}}]}k^2\rho^{k}.
$$
Since $\sigma_n^2 \geq Dn$,  
the sum in the 
numerator of \eqref{series} tends to $0$ uniformly in $k$ as  $n \rightarrow \infty$. This finishes the proof of  proposition. 
\end{proof}

\begin{lemma}\label{pre-img}
(a) The maps $\{G_n\}_{n\in \mathbb{Z}}$ for Model B satisfy property (C) 
{from  \S \ref{dens}. 

(b) for each $n \in \mathbb{Z}$ there is an interval $J_n \subseteq \mathbb{T}$ such that $G_n(J_n)=\mathbb{T}$ and 
$\DS \sup_{n \geq 1}\sup_{x \in J_n}|G_n'(x)|\leq K_0$, for some finite $K_0$. 

(c) There exists $\sigma>0$ such that}
\begin{equation}\label{mncov}
P_n \dots P_1 \mathbf{1}(x)\geq \sigma,
\end{equation}
for any $x \in \mathbb{T}$, and all $n \geq 1$.
\end{lemma}
\begin{proof}
First note that by the choice of slope $a>1$ for every $n\geq 1$ we have that the the gate $W_{n,1}$ contains a 
Markov partition element of $T_n$. Since
$G_n|_{W_{n,1}}=T_n|_{W_{n,1}}$, { we obtain part (b) 
choosing $J_n=W_{n,1}$.}

{ Since part (c) follows from parts (a) and (b) due to Proposition \ref{sptgp}(a), }
it only remains to prove property (C).

We first show that there is $k=k(|I|)$, so that
\begin{equation}\label{cov-1}
\pi_\mathbb{T}(F^k(I, 0))=\mathbb{T}.
\end{equation}

The idea of the proof of this fact is analogous to Lemma \ref{min-1} and is based on complexity estimates. Without loss of generality we can assume that $n=1$ and that $I \subset W_{1,1}$ or $I \subset W_{1,-1}$. We now construct a sequence of intervals $\{I_n\}_{n\geq 1}$ and positions $\{z_n\}_{n \geq 1}$, $z_n \in \mathbb{Z}$ $n\geq 1$, such that 
$I_n \subset W_{z_n,1}$ or $I_n \subset W_{z_n,-1}$ and $\pi_{\mathbb{Z}}F(z_n,I_n)=z_{n+1}$. $\pi_{\mathbb{T}}F(z_n,I_n)$ is divided by the singularity points $\partial W_{z_{n+1}}$ into continuity components and $I_{n+1}$ is chosen to be the largest component. Similar to Lemma \ref{min-1} we  have that
\begin{equation}\label{iter-tt}
|I_{n+1}|\geq \frac{(a-\delta_0)|I_n|}{3}.
\end{equation}
Hence, if $(a-\delta_0)>3$ then the lengths of $|I_n|$ will grow exponentially fast. Hence, for some $n$, the interval $\pi_\mathbb{T} F(z_n, I_{n})$ will cover two singularity points at once, i.e. we will have
$$
W_{z_{n+1},1}\subset\pi_\mathbb{T} F(z_n, I_{n})
\text{ or }
W_{z_{n+1},-1}\subset\pi_\mathbb{T} F(z_n, I_{n}).
$$
This means that $\pi_\mathbb{T} F(z_n, I_{n})$ will contain a gate partition in its interior.
Note also that $n<C\ln (1/|I|)$.
If now as $I_{n+1}$ we chose the corresponding gate interval and take into account that each gate contains a 
Markov partition element of $T_{z_n}$ in its interior, it will follow that $\pi_\mathbb{T} F(z_{n+1}, I_{n+1})=\mathbb{T}$, implying \eqref{cov-1}. 

Observe now that from this time onward, we can choose the interval $I_{n+k}$ to be the forward gate of $z_{n+k}$, i.e. $I_{n+k}=W_{z_{n+k,1}}$. But then $z_{n+k}=z_{n+1} + k$, for $k \geq 2$ and $\pi_\mathbb{T} F(z_{n+k}, I_{n+k})=\mathbb{T}$. Hence, there will be a time $\ell$  and a position $m> 1$ which will be visited by the walker for the first time and to reach there the walker will have to make no more then $C' \ln (1/|I|)$ steps. We will then have
$\DS 
\pi_\mathbb{T}F^\ell(1,I)=\mathbb{T}
$
and
\begin{equation}\label{(c)}
G_{1,m}(I)=\mathbb{T}.
\end{equation}
This proves the claim. Clearly $m<C \ln (1/|I|)$. Observe that by taking $m$ large we can also make the time $m$ in \eqref{(c)} uniform for all intervals $I$ with the given length. It also follows from our discussion that for every $x \in \mathbb{T}$ there exists $y \in I$ such that $G_{1,m}(y)=x$ and
$$
|G'(y)|\leq (a+\brdelta)^s,
$$
where $s \leq C' \ln (1/|I|)$. This finishes the proof of property (C). 
\end{proof}

\begin{proposition}\label{exact}
The transfer operator $\brP$ of $\brG$ satisfies property \eqref{dec}.
\end{proposition}
\begin{proof}
This follows  from Proposition \eqref{sptgp}(b), Lemma \ref{pre-img} and Proposition \ref{prop-ly}.
\end{proof}

{
\begin{proposition}\label{PrVarLin}
The variance of $\tau_n$ grows linearly.
\end{proposition}
}

\begin{proof}
 According to the discussion at the beginning of Lemma \ref{pre-img},  there exists $x_0$, such that $\brG(x_0)=x_0$ and $\brr(x_0)=1$. Then $\int_\mathbb{T}\brr hdx>\int_\mathbb{T}hdx=1=\brr(x_0)$. By  \Cref{bdd-1} 
 $\brP \brr\in BV$. Thus, by \Cref{cob}, $\brr$ is not cohomologous to zero and 
 {\Cref{lm-1}} gives the result.
\end{proof}

\section{Proof of the main results for Model B.}
\subsection{Proof of Theorem \ref{ThHit}(b).}
\label{SSHit-B}

We will check the conditions of Theorem \ref{pr-1} for the sequence $\{G_n\}_{n=1}^\infty$. 
Recall that 
$$
\tau_n(x)=\sum_{k=0}^{n-1} r_k(G_{k-1}\circ\dots \circ G_0 x). 
$$
By Proposition \ref{prop-ly}, both $\brP$ and $P_n$, $n\geq 1$ satisfy \eqref{LY} for all $\delta$ sufficiently small and $a$ large. By Proposition \ref{bdd-norm-2},  
$P_n$ and $\brP$ are close in $d_1$ norm, when $\delta$ is small. 
By Proposition \ref{exact}, $\brP$ satisfies \eqref{dec}.
Hence, by Proposition \ref{nbrhd-sc},  \eqref{dec} holds in a $d_1$ neighborhood of $\brP$. Thus, we will have \eqref{dec} for the collection $\{P_n\}_{n \in \mathbb{Z}}$ for sufficiently small $\delta$.
\eqref{Min} for $\mathcal{P}^n \mathbf{1}$ follows from Lemma \ref{pre-img}. 
Next,  \eqref{norm-bd} follows from Proposition \ref{bdd-1}. 
By \Cref{PrVarLin}
the variance of $\tau_n$ grows linearly. 
Finally,  \eqref{asm} follows from Proposition \ref{asmpt}. Now 
Theorem~\ref{ThHit}(b) follows from \Cref{pr-1}.
\qed

\subsection{Backtracking}
\label{SSBack}

\begin{lemma}\label{max}
(a) Denote $\DS z_n^{*}=\max_{0\leq k\leq n}z_k$. Then for Model  B we have 
$\DS
\lim _{n \rightarrow \infty} \frac{z_{n}^{*}-z_{n}}{\sqrt{n}}=0,
$
almost surely.

(b) There are constants $C>0, \theta<1$ such that 
$$ Pr(z_{(\cdot)}\text{ visits }n-k\text{ after reaching }
n)\leq C \theta^k$$
where $\Pr$ denotes the Lebesgue measure.
\end{lemma}
\begin{proof}
(a) Without the loss of generality we can assume $z_0=0$. 
By Borel-Cantelli Lemma it suffices to show that for each $t>0$
$$
\sum_{n=1} \Pr\left(x: \frac{z_{n}(x)^{*}-z_{n}(x)}{\sqrt{n}}>t\right)<\infty.
$$

Let $\ell_n(x)=\min \{k:0\leq k\leq n, z_k(x)=z_n^*(x)\}$. One can see that if $z_{n}^{*}-z_{n}>t\sqrt{n}$ then $\ell_n(x)\leq n-t\sqrt{n}$. Hence
\begin{equation}\label{estm1}
\Pr(z_{n}^{*}-z_{n}>t\sqrt{n})\leq  \Pr(\ell_n(x)\leq n-t\sqrt{n}+1)
=\sum_{k=1}^{[n-t\sqrt{n}]+1}\Pr(\ell_n(x)=k).
\end{equation}
Next consider the sets
$$
B_{m,k}=\{x:z_k(x)\leq m| z_0=m\}.
$$
This is the set of points, for which the walker starting its walk at $z_0=m$ will be located to the left of $m$ after $k$ steps.
{By \eqref{LongExcursion}}
we have that for every $m \in \mathbb{Z}, k\geq 1$
$$
|B_{m,k}|\leq C\rho^{k},
$$
for some $\rho<1$.
Next, note that
\begin{align*}
\Pr(\ell_n(x)=k)&\leq \sum_{s=1}^k\int_\mathbb{T}\chi_{B_{s,n-k}}(G_{s-1}\circ \cdots\circ G_0(x))dx \\
& =\sum_{s=1}^k \int_\mathbb{T}\chi_{B_{s,n-k}}(x)P_{s-1}\dots P_0 \mathbf{1} dx\leq M k\sup_{s \leq k}|B_{s,n-k}|\leq CM k\rho^{n-k}
\end{align*}
{ where $M$ is from \eqref{den-bdd}.}
Hence, by \eqref{estm1} and due to $k \leq n-t\sqrt{n}+1$, we can write
$$
\Pr(z_{n}^{*}-z_{n}>t\sqrt{n})\leq 
n CM n\rho^{t\sqrt{n}-1}.
$$
To finish the proof { of part (a)} it is enough to notice that
$$
\DS
\sum_{n \geq 1} \Pr\left(z_{n}^{*}-z_{n} \geq t\sqrt{n}\right)<MC\sum_{n \geq 1}
n^2 \rho^{t \sqrt{n}-1}
<\infty.
$$

{ Part (b) also follows from \eqref{LongExcursion} and the fact that upon first time reaching level n the internal
state of the walker is distributed with the bounded density $\cP^n \bf{1}.$}
\end{proof}

{
\subsection{Proof of Theorem \ref{main} for Model B} 
\label{ScNatural}
Given the results of \S\S \ref{SSHit-B}--\ref{SSBack} the proof of Theorem \ref{main} for Model B is similar to
the proof for Model A and requires only minor modifications which we presently describe.

(1) \eqref{Z-Z*} no longer holds, however \Cref{max}(a)  is sufficient for replacing $z_n^*$ by $z_n$ in 
our limit theorems.
\\

(2) The proof of \Cref{LmQExp}(c) needs to be modified since 
$\fa_{n_1,n_1-n_2}$ and $\fa_{n_2}$ are no longer independent.
However, one can replace
$\fa_{n,k}$ by 
$$ \tilde\fa_{m, k}=\int_\mathbb{T} r_m(\tG_{m-1} \circ \dots \circ \tG_{m-k/2} (x))dx$$
where $\tG_\ell$ are obtained by motion in the environment where $W_{m-k, 1}=\T$.
In other words, upon reaching level $m-k$ the particle makes the next step to the right with probability 1.
Then $\tilde\fa_{n_1,n_1-n_2}$ and $\fa_{n_2}$ are independent. On the other hand
$$ \fa_m-\tilde\fa_{m,k}=[\fa_m-\fa_{m,k/2}]+[\fa_{m, k/2}-\tilde\fa_{m,k}]. $$
The first term is exponentially small due to \eqref{dec}, while the second term equals
to 
$$ \int_\mathbb{T}[ r_m(\tG_{m-1} \circ \dots \circ \tG_{m-k/2} (x))- r_m(G_{m-1} \circ \dots \circ G_{m-k/2} (x))]
dx$$
and it is exponentially small since the integrand is non-zero only if the walker starting from level 
$m-k/2$ backtracks to level $m-k$ which happens with exponentially small probability due to
\eqref{LongExcursion}. \\

(3) It is no longer true that $\tau_n\leq 3n$ so the proof of \Cref{LmLin} needs to be modified. 

As before the estimate for $b_n$ follows from the estimate for $\cS(n).$ The lower bound on $\cS(n)$ 
still holds because $\tau_n\geq n$. The upper bound follows from the uniform integrability of $r_m$
which is ensured by \eqref{LongExcursion}.

The lower bound on $\brsigma_n^2$ follows from \Cref{PrVarLin}
while the upper bound follows from \eqref{dec}. Namely, while it is no longer true for Model B that
$r_m\in BV$ the fact
 that $P_{m} r_m$ is uniformly bounded in $BV$ suffices to get the exponentially decay
of $\mathrm{Cov}(r_{n_1}, r_{n_2})$. Indeed denoting 
$\DS \tilde r_k(x)=r_k(x)-\int_\T r_k(G_{k-1}\circ\dots\circ G_0 y) dy$
we get that for $n_1>n_2$ 
$$ \int_T \tilde r_{n_1}(G_{n_1-1}\circ\dots\circ G_0 x)
\tilde r_{n_2}(G_{n_2-1}\circ\dots\circ G_0 x) dx$$
$$=
\int_\mathbb{T} \tilde{r}_{n_1} P_{n_1-1} \cdots P_{n_2}\left(\tilde{r}_{n_2} \mathcal{P}^{n_2} 1\right) d x
$$
which is exponentially small since $P_{n_2}\left(\tilde{r}_{n_2} \mathcal{P}^{n_2} 1\right)\in BV_0.$

With the changes (1)--(3) discussed above the proof of Theorem \ref{main} for Model B proceeds by
the same arguments as for Model A.

}

\appendix    
\section{Sequential CLT for unbounded observables.}\label{app}


\begin{proof}[Proof of Theorem \ref{pr-1}]

Note that in general $\tilde{f}_n\notin BV$, but by \eqref{norm-bd} and \eqref{den-bdd}
$$
|P_k(f_{k-1}\mathcal{P}^{k-1}\mathbf{1})|_{BV}\leq D  |\mathcal{P}^{k-1}\mathbf{1}|_{BV}\leq MD.
$$
Hence, by \eqref{dec}
$$
|P_{n} P_{n-1}\dots P_{n-k} \left(\tilde{f}_{n-k-1} \mathcal{P}^{n-k-1} \mathbf{1}\right)|_{BV} \leq K\theta^{k}MD.
$$
Thus, we can consider the martingale co-boundary decomposition defined in \cite{CR07}
\begin{equation}\label{dc-0}
\mathbf{H}_{n}=\frac{1}{\mathcal{P}^{n} \mathbf{1}}\left[P_{n}\left(\tilde{f}_{n-1} \mathcal{P}^{n-1} \mathbf{1}\right)+P_{n} P_{n-1}\left(\tilde{f}_{n-2} \mathcal{P}^{n-2} \mathbf{1}\right)+\cdots+P_{n} P_{n-1} \ldots P_{1}\left(\tilde{f}_{0} \mathcal{P}^{0} \mathbf{1}\right)\right]
\end{equation}
and set 
\begin{equation}\label{dcth-1}
\begin{aligned} \psi_{n} &=\tilde{f}_{n}+\mathbf{H}_{n}-\mathbf{H}_{n+1}\circ G_{n+1}  \\ U_{n} &= \psi_{n}(G_{n} \circ \ldots \circ G_{1}). \end{aligned}
\end{equation}
Clearly
$$
|\mathbf{H}_{n}|_{BV} \leq \frac{KMD}{\sigma} \sum_{j=1}^n \theta^k
$$
{where $\sigma$ is from \eqref{Min}.}
Since $\mathcal{P}^n \mathbf{1}\in BV$, 
and by \eqref{Min} we also have $\frac{1}{\mathcal{P}^n \mathbf{1}}\in BV$, then $\mathbf{H}_n \in BV$, for all $n \geq 1$.
Moreover
$$
\sup_n|\mathbf{H}_{n}|_{BV} < \infty.
$$ 
Next note that
$$
|P_{n}\left(\psi_{n-1} \mathcal{P}^{n-1} \mathbf{1}\right)|_{BV} \leq |P_{n}\left(\tilde{f}_n \mathcal{P}^{n-1} \mathbf{1}\right)|_{BV}+ |P_n\left((\mathbf{H}_{n}- \mathbf{H}_{n+1}\circ G_{n+1}\right)\mathcal{P}^{n-1}\mathbf{1})|_{BV}<\infty.
$$
It then follows that
\begin{equation}\label{est-2}
\sup_{n \geq 1}|P_{n}\left(\psi_{n-1} \mathcal{P}^{n-1} \mathbf{1}\right)|_{BV}<\infty, \quad 
\sup_{n \geq 1}|P_{n}\left(\psi^2_{n-1} \mathcal{P}^{n-1} \mathbf{1}\right)|_{BV}<\infty.
\end{equation}
It is shown in \cite{CR07}, that $U_n$ is a sequence of reversed martingale and one has
$$
\sum_{k=0}^{n-1} \tilde{f}_{k}(G_1^k)=\sum_{k=0}^{n-1} U_{k}(G_1^k)+\mathbf{H}_{n}(G_1^k).
$$
We recall the following estimate from \cite{CR07}
\begin{equation}\label{var-bdd}
\Big| \|S_{n}\|_{2}-\|\sum_{k=0}^{n-1} U_{k}\|_{2}\Big| =\Big| \|S_{n}\|_{2}-\left(\sum_{k=0}^{n-1} 
\int U_{k}^{2}(x)  dx\right)^{\frac{1}{2}}\Big|
\end{equation}
$$
\leq \|S_n-\sum_{k=0}^{n-1} U_{k} \|_{2}\leq \sup_{n \geq 1}|\mathbf{H}_{n}|_{BV}<\infty. 
$$
Thus, $\sigma_n$ is unbounded if and only if 
$\DS \sum_{k=0}^{n-1} 
\int U_{k}^{2}(x)  dx$ is. The last expression is monotone. Hence, if $\sigma_n$ is unbounded, then it has to tend to infinity as $n \rightarrow \infty$.

We now define
$$
\bar\sigma_{n}^{2}=\sum_{k=0}^{n-1} \int U_{k}^{2}dx, \quad V_{n}=\sum_{k=0}^{n-1} \int\left[U_{k}^{2} | \mathcal{A}_{k+1}\right]dx.
$$
Following \cite{CR07}, Theorem 5.1, we need to check the following two conditions of Theorem 5.8 of \cite{CR07}, which is an extension of a result of B.M. Brown \cite{Br71}

\begin{enumerate}[(i)]
  
  \item   for every $\varepsilon>0$,\;
$\DS
\lim _{n \rightarrow+\infty} \bar\sigma_{n}^{-2} \sum_{k=0}^{n-1} \int\left[U_{k}^{2} 1_{\left\{\left|U_{k}\right|>\varepsilon \sigma_{n}\right\}}\right]=0.
$
\item the sequence $\left(\bar\sigma_{n}^{-2} V_{n}\right)_{n \geq 1}$ converges to $1$ in probability.
\end{enumerate}

For (i) we have from \cite{CR07}, page 115 and the estimate
$\|\mathcal{P}^{n}\mathbf{1}\|_\infty\leq M$, that for all $n \geq 1$
\begin{align*}
\int U_{k}^{2} 1_{\left\{\left|U_{k}\right|>\varepsilon \sigma_{n}\right\}}dx &=\int\left[\psi_{k}^{2}(G_1^n) 1_{[\varepsilon \sigma_{n}, \infty)}(\psi_{k}^{2}(G_1^n))\right]dx\\
& = \int\left[\psi_{k}^{2}(x) 1_{[\varepsilon \sigma_{n}, \infty)}(\psi_{k}^{2}(x))\mathcal{P}^{n} \mathbf{1}\right]dx
 \leq M \int\left[\psi_{k}^{2}(x) 1_{[\varepsilon \sigma_{n}, \infty)}(\psi_{k}^{2}(x))\right]dx.
\end{align*}
By \eqref{dcth-1}, \;
$\DS \psi_{k}^{2}(x)\leq 2(\tilde{f}_{k}^{2}(x))^2 + 8 \sup_{k \geq 1}|\mathbf{H}_{k}|_\infty^2$. Hence, if $\varepsilon \bar\sigma_n\leq \psi_{k}^{2}(x)$, then for $n$ large
$$
\frac{\varepsilon \bar\sigma_n}{8}\leq (\tilde{f}_{k}^{2}(x))^2.
$$ 	
Take $n$ so large that $\DS 4 \sup_{k \geq 1}|\mathbf{H}_{k}|_\infty^2 \leq \frac{\varepsilon \bar\sigma_{n}}{8}$.
Then
$$
M\int\left[\psi_{k}^{2}(x) 1_{[\varepsilon \bar\sigma_{n}, \infty)}(\psi_{k}^{2}(x))\right]dx\leq 2M \int\left[(\tilde{f}_{k}^{2}(x) + 4 \sup_{k \geq 1}|\mathbf{H}_{k}|_\infty^2) 1_{[\frac{\varepsilon \bar\sigma_{n}}{8} , \infty)}(\tilde{f}_{k}^{2}(x))\right]
$$
$$
\leq 4M \int\left[\tilde{f}_{k}^{2}(x) 1_{[\frac{\varepsilon \bar\sigma_{n}}{8} , \infty)}(\tilde{f}_{k}^{2}(x))\right].
$$
Hence (i) follows from \eqref{asm}. 

As for (ii) we have by \cite{CR07}, page 115
$$
\int\left[U_{k}^{2} | \mathcal{A}_{k+1}\right]=
\left(\frac{P_{k+1}\left(\psi^{2}_k \mathcal{P}^{k} 1\right)}{\mathcal{P}^{k+1} 1}\right)
 \circ
(G_{n} \circ \cdots \circ G_{k+1})
$$
By \eqref{Min} and \eqref{est-2} we have that
$$
\sup _{k}\left|\left(\frac{P_{k+1}\left(\psi^{2}_k \mathcal{P}^{k} 1\right)}{\mathcal{P}^{k+1} 1}\right)\right|_{BV}<\infty.
$$
{ Given these estimates,} the rest of the proof { of (ii)} 
is the same as { in the proof of \cite[Theorem 5.1]{CR07}.}
\end{proof}

{To verify the growth of the variance assumption in Theorem \ref{pr-1}, the following fact will be helpful.}

\begin{proposition}\label{A-2}
Let $\brG$ be such that $\brP$ satisfies \eqref{dec} and for its acim we almost surely have that $h(x)\geq c>0$. 
Assume that $P\bar\tau \in BV$. Let
$$
\overline{\mathbf{H}}=\frac{1}{h} \sum_{n=1}^{\infty} P^{n}\left(h\left({\bar\tau}
-\int {\bar\tau} h d x\right)\right),
$$
Then  $\overline{\mathbf{H}}\in BV$.
Moreover, {if 
\begin{equation}\label{gording}
\psi := \bar\tau - \int_\mathbb{T} \bar\tau hdx + \overline{\mathbf{H}} - \overline{\mathbf{H}}\circ \brG,
\end{equation}
does not vanish almost surely} then
$$\hat\sigma^2_n= \var\Big[\sum_{k=1}^n \bar\tau { \circ }\brG^k\Big] \geq Cn$$ for some $C>0.$
\end{proposition}
\begin{proof}
Consider the coboundary decomposition { from the proof of} Theorem \ref{pr-1}
We take $P_n=\brP$ and $f_n=\tau$ for all $n \geq 1$. Then
$$
\mathbf{H}_{n, P}=\frac{1}{\brP^{n} \mathbf{1}}\left[\brP\left(\tilde{\tau}\brP^{n-1} \mathbf{1}\right)+\brP^2\left(\tilde{\tau} \brP^{n-2} \mathbf{1}\right)+\cdots+\brP^n\left(\tilde{\tau} \brP \mathbf{1}\right)\right].
$$
and respectively
$$
\psi_n = \tilde{\tau} + \mathbf{H}_{n, P} - \mathbf{H}_{n, P}\circ G.
$$
We now show that
\begin{equation}\label{cbd-l}
\mathbf{H}_{n, P} \rightarrow_{L^2} \overline{\mathbf{H}}
\end{equation}
as $n \rightarrow \infty$. For this note that $\int_\mathbb{T} h\left(\tau-\int \tau h d x\right)dx=0$. Hence
$$
\Big|P^{n}\left(h\left(\tau-\int \tau h d x\right)\right)|_{BV} \leq K\theta^{n-1}\Big|P\left(h\left(\tau-\int \tau h d x\right)\right)\Big|_{BV}.
$$
Thus the general term in \eqref{cbd-l} decays exponentially fast. For small values of $n$ the convergence follows from the fact
$\brP^n \mathbf{1}\rightarrow_{L^2} h$, as $n \rightarrow \infty$, and the continuity of $\brP$ in $L^2$ metric. 
We also have that $\frac{1}{\brP^{n}\mathbf{1}}\rightarrow_{L^2}\frac{1}{h}$. Thus, \eqref{gording} follows.

Now assume that $\|\psi\|_{2}>0$.
Then by \eqref{var-bdd} 
$$
\Big| \hat\sigma_n-\left(\sum_{k=0}^{n-1} 
\int_\mathbb{T} U_{k}^{2}(x)  dx\right)^{\frac{1}{2}}\Big|=\Big| \hat\sigma_n-\left(
n \int_\mathbb{T}\psi^2 hdx\right)^{\frac{1}{2}}\Big|\leq |\overline{\mathbf{H}}|_{BV}<\infty
$$
completing the proof of the proposition.
\end{proof}

\newcommand{\Addresses}{{
  \bigskip
  \footnotesize

  \medskip

}}

\maketitle

\Addresses


\begin{thebibliography}{999}

\bibitem{AL18} Aimino R., Liverani C.
{\em Deterministic walks in random environment,}
Ann. Prob. {\bf 48} (2020) 2212--2257.

\bibitem{AR14}
Aimino R., Rousseau J. {\em Concentration inequalities for sequential dynamical systems of the unit interval,} 
Ergodic Theory Dynam. Systems {\bf 36} (2016) 2384--2407. 

\bibitem{Br71}
Brown B. M. {\em Martingale central limit theorems}, Ann. Math. Stat. {\bf 42}
(1971) 59--66.

\bibitem{BDR} Berger N., Drewitz A. Ramirez A. F. 
{\em Effective polynomial ballisticity conditions for random walk in random environment,} 
Comm. Pure Appl. Math. {\bf 67} (2014) 1947--1973.

\bibitem{Bi11} Biskup M. {\em Recent progress on the random conductance model,} 
Prob. Surveys {\bf 8} (2011) 294--373.

\bibitem{BBS} Boldrighini C., Bunimovich L. A.; Sinai Ya. G. 
{\em On the Boltzmann equation for the Lorentz gas,} 
J. Statist. Phys. {\bf 32} (1983) 477--501.

\bibitem{BG00} Bolthausen E., Goldsheid, I. 
{\em Recurrence and transience of random walks in random environments on a strip,} 
Comm. Math. Phys. {\bf 214} (2000)  429--447. 

\bibitem{BG08} Bolthausen E., Goldsheid, I. 
{\em Lingering random walks in random environment on a strip,} Comm. Math. Phys. 
{\bf 278} (2008)  253--288. 


\bibitem{CR07} Conze J.--P., Raugi A. {\em Limit theorems for sequential expanding dynamical systems on $[0,1]$,} 
Contemp. Math. {\bf 430} (2007) 89--121.

\bibitem{CPSV09} Comets F., Popov S., Schutz G.M., Vachkovskaia M. 
{\em Billiards in a general domain with random reflections,} 
Arch. Ration. Mech. Anal. {\bf 191} (2009) 497--537.

\bibitem{CPSV10a} Comets F., Popov S., Schutz G.M., Vachkovskaia M. 
{\em Quenched invariance principle for Knudsen stochastic billiard in random tube,} 
Ann. Probab. {\bf 38} (2010) 1019--1061.

\bibitem{CPSV10b} Comets F., Popov S., Schutz G.M., Vachkovskaia M. 
{\em Knudsen gas in a finite random tube: transport diffusion and first passage properties,} 
J. Stat. Phys. {\bf 140} (2010) 948--984.

\bibitem{CL10}  Cristadoro G., Lenci M., Seri M. 
{\em Recurrence for quenched random Lorentz tubes,} 
Chaos {\bf 20} (2010) paper 023115, 7 pp.

\bibitem{DL21} Demers M. F., Liverani C.
{\em Projective cones for generalized dispersing billiards,}
arXiv:2104.06947.

\bibitem{DG13}  Dolgopyat D., Goldsheid I. 
{\em Limit theorems for random walks on a strip in subdiffusive regimes,} Nonlinearity {\bf 26} (2013) 1743--1782.

\bibitem{DG20}  Dolgopyat D., Goldsheid I. 
{\em Local limit theorems for random walks in a random environment on a strip,} 
Pure Appl. Funct. Anal. {\bf 5} (2020) 1297--1318.

\bibitem{DK09} Dolgopyat D., Koralov L. 
{\em Motion in a random force field,} Nonlinearity {\bf 22} (2009) 187--211.


\bibitem{DFGTV} Dragicevic D., Froyland G., Gonzalez-Tokman C., Vaienti S. 
{\em A spectral approach for quenched limit theorems for random expanding dynamical systems,} 
Comm. Math. Phys. {\bf 360} (2018) 1121--1187.

\bibitem{Fe07} Feres R. 
{\em Random walks derived from billiards,} MSRI Publ. {\bf 54} (2007) 
179--222.

\bibitem{Ga82} V. P. Gaposhkin: {\it On the Dependence of the Convergence Rate in the Strong
Law of Large Numbers for Stationary Processes on the Rate of Decay of the
Correlation Function,} Th. Prob., Appl. {\bf 26} (1982) 706--720.

\bibitem{Go08} Goldsheid I. Ya. 
{\em Linear and sub-linear growth and the CLT for hitting times of a random walk in random environment on a strip,} Probab. Th. Rel. Fields {\bf 141} (2008) 471--511.

\bibitem{HNTV} Haydn N., Nicol M., Torok A., Vaienti S. 
{\em Almost sure invariance principle for sequential and non-stationary dynamical systems,} 
Trans. AMS {\bf 369} (2017) 5293--5316. 

\bibitem{L-Y}
A. Lasota, J.-A. Yorke, {\em On the existence of invariant measures for piecewise monotonic transformations,}
Trans. AMS {\bf 186} (1973)  481--488.


\bibitem{I-L} Ibragimov I. A., Linnik, Yu. V. 
{\em Independent and stationary sequences of random variables.,}
  Wolters-Noordhoff Publishing, Groningen (1971) 443 pp.
  
\bibitem{KKS} Kesten H., Kozlov M. V., Spitzer F. 
{\em A limit law for random walk in a random environment,} 
Compos. Math. {\bf 30} (1975) 145--168.

\bibitem{KP79} Kesten H., Papanicolaou G. C.
{\em A limit theorem for turbulent diffusion,}
Comm. Math. Phys. {\bf 65} (1979) 97--128.

\bibitem{KP81} Kesten H., Papanicolaou G. C.
{\em A limit theorem for stochastic acceleration,} Commun. Math. Phys. 
{\bf 78}  (1981) 19--63. 
\bibitem{Kh1} Khasminskii R. Z. 
{\em Stochastic processes defined by differential equations with a small parameter,} 
Th. Probab. Appl. {\bf 11} (1966) 211--228. 

\bibitem{Kh2} Khasminskii R. Z. 
{\em A limit theorem for solutions of differential equations with a random right hand part,} 
Th. Probab. Appl.
{\bf 11} (1966) 444--462. 

\bibitem{K98} Kifer Y.
{\em Limit theorems for random transformations and processes in random environments,} 
Trans. AMS {\bf 350} (1998) 1481--1518.


\bibitem{KR1} Komorowski T., Ryzhik L 
{\em Diffusion in a weakly random Hamiltonian flow,} Comm. Math. Phys. 
{\bf 263} (2006) 277--323.

\bibitem{KR2} Komorowski T., Ryzhik L 
{\em The stochastic acceleration problem in two dimensions,} Israel J. Math. {\bf 155} 
(2006) 157--204.



\bibitem{LT11} Lenci M.,  Troubetzkoy S.
{\em Infinite-horizon Lorentz tubes and gases: recurrence and ergodic properties,}
Phys. D {\bf 240} (2011) 1510--1515.

\bibitem{Liv1} Liverani C. {\em Decay of correlations for piecewise expanding maps}, J. Stat. Phys.
{\bf  78} (1995) 1111--1129.

\bibitem{LuT} 
Lutsko C., T\'oth B. {\em Invariance principle for the random Lorentz gas beyond the Boltzmann-Grad limit,} 
Comm. Math. Phys. {\bf 379} (2020) 589--632.


\bibitem{M-T-V}
Nicol M., T\"or\"ok A., Vaienti S.
{\em Central limit theorems for sequential and random intermittent dynamical systems,} 
Erg. Th. Dyn. Sys. {\bf 38} (2018) 1127--1153.

\bibitem{NSV12} Nandori P.,  Szasz D., Varju T. 
{\em A central limit theorem for time-dependent dynamical systems,} 
J. Stat. Phys. {\bf 146} (2012) 1213--1220. 

\bibitem{PV} Papanicolaou G. C., Varadhan S. R. S.
{\em Diffusions with random coefficients,} in {\em Statistics and probability: essays in honor of C. R. Rao,}
Ed. G. Kallianpur, Paruchuri R. Krishnaiah and J. K. Ghosh. 
North-Holland Publishing Co., Amsterdam-New York, 1982, pp. 547--552.

\bibitem{S82} Sinai Ya. G. {\it The limiting behavior of a one-dimensional
random walk in a random medium,} Theory Prob. Appl. \textbf{27} (1982)
256--268.

\bibitem{Sol} Solomon F. {\em Random walks in a random environment,} Ann. Probab. {\bf 3}
(1975) 1--31. 

\bibitem{Sp80} Spohn H. 
{\em Kinetic equations from Hamiltonian dynamics: Markovian limits,} Rev. Modern Phys. {\bf 52} (1980)
569--615.

\bibitem{Sp91} Spohn H. 
{\em Large scale dynamics of interacting particles,} Springer, Berlin-New York, 1991.

\bibitem{Szn01} Sznitman A.--S. 
{\em On a class of transient random walks in random environment,} 
Ann. Prob. {\bf  29} (2001) 724--765. 

\bibitem{Szn02} Sznitman A.--S. 
{\em An effective criterion for ballistic behavior of random walks in random environment,}
Prob. Th. Rel. Fields {\bf 122}  (2002) 509--544.

\bibitem{SznZ} Sznitman A.--S., Zeitouni O. 
{\em An invariance principle for isotropic diffusions in random environment,} 
Invent. Math. {\bf 164} (2006) 455--567.

\bibitem{Viana}
Viana M.
{\em Stochastic dynamics of deterministic systems,}
Brazillian Math. Colloquium (1997) IMPA, 197 pp.

\end{thebibliography}
\end{document}